\documentclass[11pt,reqno]{article}
\usepackage{amsthm,amsmath, amssymb, amsopn, amsfonts}
\usepackage[margin=1in]{geometry}
\usepackage[colorlinks,citecolor=blue,urlcolor=blue]{hyperref}
\usepackage{graphicx}
\usepackage{tikz}
\usetikzlibrary{decorations.fractals}
\usetikzlibrary{patterns,arrows,shapes,decorations.pathreplacing}
\usepackage{pgfplots}
\usepackage{stmaryrd} % for stochastic interval \llbracket
\usepackage{soul,xcolor}
\setstcolor{red}

\numberwithin{equation}{section}
\theoremstyle{plain}
\newtheorem{Definition}{Definition}[section]
\newtheorem{Remark}{Remark}[section]
\newtheorem{Theorem}{Theorem}[section]
\newtheorem{Lemma}{Lemma}[section]
\newtheorem{Proposition}{Proposition}[section]
\newtheorem{Corollary}{Corollary}[section]
\newtheorem{Assumption}{Assumption}[section]
\newtheorem{Example}{Example}[section]
\newcommand{\be}{\begin{equation}}
\newcommand{\ee}{\end{equation}}
\newcommand{\bee}{\begin{equation*}}
\newcommand{\eee}{\end{equation*}}
\newcommand{\bi}{\begin{itemize}}
\newcommand{\ei}{\end{itemize}}
\def \E{\mathbb{E}}

\def \N{\mathbb{N}}
\def \P{\mathbb{P}}

\def \R{\mathbb{R}}
\def \X{\mathbb{X}}

\def \Bc{{\mathcal B}}
\def \Ec{{\mathcal E}}

\def \Gc{{\mathcal G}}

\def \eps{\varepsilon}

\newcommand{\T}{\mathcal{T}}

\makeatletter
\newcommand{\setword}[2]{%
	\phantomsection
	#1\def\@currentlabel{\unexpanded{#1}}\label{#2}%
}
\makeatother

\title{Stability of Equilibria in Time-inconsistent Stopping Problems}

\author{Erhan Bayraktar\thanks{
		Department of Mathematics, University of Michigan, Ann Arbor, email: \texttt{erhan@umich.edu}. E. Bayraktar is partially supported by the National Science Foundation under grant DMS2106556 and by the Susan M. Smith chair.}	
	\and Zhenhua Wang\thanks{
		Department of Mathematics, University of Michigan, Ann Arbor, email: \texttt{zhenhuaw@umich.edu}. }
	\and Zhou Zhou\thanks{School of Mathematics and Statistics, University of Sydney, Australia, email:
		\texttt{zhou.zhou@sydney.edu.au}.}
}

\begin{document}

\maketitle

\date{}

\begin{abstract}
We investigate the stability of equilibrium-induced optimal values with respect to (w.r.t.) reward functions $f$ and transition kernels $Q$ for time-inconsistent stopping problems under nonexponential discounting in discrete time. First, with locally uniform convergence of $f$ and $Q$ equipped with total variation distance, we show that the optimal  value is semi-continuous w.r.t. $(f,Q)$. We provide examples showing that continuity may fail in general, and the convergence for $Q$ in total variation cannot be replaced by weak convergence.  Next we show that with the uniform convergence of $f$ and $Q$, the optimal value is continuous w.r.t. $(f,Q)$ when we consider a relaxed limit over $\varepsilon$-equilibria. We also provide an example showing that for such continuity the uniform convergence of $(f,Q)$ cannot be replaced by locally uniform convergence.
\end{abstract}	

{
\hypersetup{linkcolor=black}
\tableofcontents
}
	
\section{Introduction}	
Consider the optimal stopping problem 
\be\label{eq.intro.optimalstopping} 
\sup_{\tau\in \T} \E_x[\delta(\tau)f(X_\tau)],
\ee
where $X=(X_t)_{t\in[0,\infty)}$ is a time-homogeneous Markov process taking values in some state space $\X$, $\delta$ is a discount function and $f$ is a reward function. It is well known that when $\delta$ is not exponential, the problem \eqref{eq.intro.optimalstopping} may be time-inconsistent. That is, a stopping strategy that is optimal from today's point of view may no longer be optimal from a future's perspective. A popular approach to address this time-inconsistency is to look for a subgame perfect Nash equilibrium instead of solving \eqref{eq.intro.optimalstopping}: a strategy such that once it is imposed over the planning horizon, the current self has no incentive to deviate from the strategy, given all future selves will follow it.

There have been a lot of papers on equilibrium strategies for time-inconsistent control problems, and we refer to \cite{MR4288523, MR4328502, MR3626618} and the references therein. The development for theory of time-inconsistent stopping is more recent, and we refer to \cite{MR4250561,huang2018time,MR4067078,MR4273542,MR4116459,MR4080735,MR3880244,MR4205889,MR4332966,MR3980261,2022arXiv220107659B,MR3911711}.
Let us also mention the work \cite{MR4397932} which analyzes a time-inconsistent Dynkin game, and \cite{liang2021weak} which considers a time-inconsistent controller-stopper problem. It is worth to mention that most of the papers on time-inconsistent control and stopping focus on the characterization of equilibria. A few exceptions include \cite{MR3911711,MR4116459,MR4250561,MR4121091} where the optimality and selection of equilibria are first analyzed in the presence of multiple equilibria. In particular, it is shown in settings of these papers that there exists an optimal equilibrium which pointwisely dominates all other equilibria in terms of the associated value functions; moreover, this optimal equilibrium is given by the intersection of all equilibria and thus the smallest equilibria.

%
%which are consistent planning, are studied in the literature. For different discussions and approaches, see \cite{MR4205889}, \cite{MR3880244}, \cite{MR4328502}, \cite{huang2018time}, and there references therein.

The focus of this paper differs from those in the existing literature on time-inconsistent problems: we consider the stability of (smallest optimal) equilibria as well as the optimal values induced by these equilibria (or by the optimal equilibria). More specifically, we investigate the continuity of the optimal equilibrium and optimal value with respect to (w.r.t.) the reward function $f$ and the transition kernel $Q$ of the Markov process $X$. Our first main result, Theorem \ref{t1}, states that, with the local convergence of $f$ and $Q$ which is equipped with the total variation distance, the optimal equilibria (in terms of inclusion) is lower semicontinuous, and the optimal value function is upper semicontinuous w.r.t. $(f,Q)$. We provide examples showing that the exact continuity w.r.t. $(f,Q)$ for either the optimal equilibrium or the optimal value function may fail. Moreover, we also construct an example in which the semi-continuity fails if the convergence of $Q$ in total variation is changed to weak convergence. Let us emphasize that our first main result contrasts with the stability of the optimal value w.r.t. $(f,Q)$ under time-consistent stopping (i.e., with exponential discounting): the continuity indeed holds for time-consistent stopping in our setup, as indicated in Remark \ref{rm.exponential}.

In our second main result, Theorem \ref{thm.Qf.continuity}, we recover the continuity (under a relaxation) of the optimal value function w.r.t. $(f,Q)$ by relaxing the equilibrium concept and including $\eps$-equilibria: Specifically, we show that as $(f^n,Q^n)$ uniformly converges to $(f,Q)$, it holds that $\lim_{\eps\searrow 0}\lim_{n\to\infty}V_\eps^{Q^n}(\cdot,f^n)=V_0^{Q}(\cdot,f)$, where $V_\eps^{Q^n}(\cdot,f^n)$ is the optimal value induced by all $\eps$-equilibria w.r.t. $(f^n,Q^n)$. The two limits in $\eps$ and $n$ cannot be changed due to the first main result; see Remark \ref{rm.exchange.application}. To prove the second main result, we introduce the notion of pseudo $\eps$-equilibrium which captures the idea of penalizing the possible deviation in the continuation region but not in the stopping region; see Definition \ref{def.equi.pseudoepsi}. It turns out that pseudo $\eps$-equilibria have better properties than $\eps$-equilibria: One can embed the set of pseudo-$\eps$-equilibria to pseudo equilibria corresponding to a perturbed reward function; see Lemma~\ref{lm.equi.epsilon}.
A remarkable observation is that the smallest optimal pseudo equilibrium is actually the smallest optimal equilibrium; see Proposition~\ref{prop.optimalequi.pseudo}. In Example~\ref{ex:countersecond}, we demonstrate that the continuity in our second main result may fail if we replace the uniform convergence of $(f,Q)$ with locally uniform convergence. In Proposition~\ref{prop.continue.pseudo}, however, we show that if the relaxation is over the pseudo $\eps$ equilibria, then the uniform convergence can be loosened.

%This paper aims to study the stability of equilibria and optimal value function $V^Q(x,f)$ over all equilibria for time-inconsistent optimal stopping problems in the discrete time setting.  In general, equilibrium (see Definition \ref{def.equilibrium}) is not unique, and the intersection of all equilibria, $S^*(f,Q)$ (defined in \eqref{eq.value.optima}), forms an optimal equilibrium (see Lemma \ref{lm.iteration.sstar}), and the optimal value function $V^Q(x,f)$ over all equilibria (see \eqref{eq.value.optima}) is achieved by the pay-off function of $S^*(f,Q)$. Following that, we study the stability of the optimal value function $V^Q(x,f)$ w.r.t the process law and reward function. 
 
Stability analysis is an important topic in control and optimization problems. For the stability of equilibria, let us mention the very recent works \cite{feinstein2020continuity} and \cite{feinstein2022dynamic} on Nash games. To the best of our knowledge, there is no literature so far studying the stability of equilibria for time-inconsistent (stopping) problems. In this regards, our paper provides very novel and conceptual contributions to the stability analysis in the topic of time-inconsistent problems. Our results also give a theoretical guidance for the numerical computation of optimal equilibrium values for time-inconsistent stopping: with good estimation of the reward function $f$ and transition kernel $Q$, one needs to use $\eps$-equilibria instead of perfect equilibria to estimate the optimal value induced by perfect equilibria.

The rest of the paper is organized as follows. The setup and main assumptions are introduced in Section \ref{sec:set.up}, together with several preliminary lemmata. In Section \ref{sec:upper.continue}, we present our first main result, the proof of which is given in Section \ref{subsec:thm1.proof}. In Section \ref{sec:continue.epsilon}, we provide the second main result by introducing (pseudo) $\varepsilon$-equilibria. The proof of this result is collected in Section \ref{subsec:proof.continue}. Appendix gathers the proofs of lemmata in Section \ref{sec:set.up}.
	
\section{Setup and preliminaries}\label{sec:set.up}	
Consider a measurable space $(\Omega,\mathcal{F})$ and let $X=(X_t)_{t=0,1,\dotso}$ be a time-homogeneous Markov process in discrete time, taking values in some polish space $\X$. Let $\mathbb{F}$ be the filtration generated by $X$. Denote $\mathcal{B}$ the class of Borel sets of $\X$, and $\mathbb{N}:=\{0,1,2,\dotso\}$, $\overline\N:=\N\cup\{\infty\}$, $\R_+:=[0,\infty)$. Let $f:\X\rightarrow \R_+$ be a reward function that may be discontinuous. Denote $||f||_\infty:=\sup_{x\in\X}|f(x)|$. Let $\delta:\mathbb{N}\mapsto[0,1]$ be a discount function that is decreasing with $\delta(0)=1$, $\delta(1)<1$ and $\lim_{t\to\infty}\delta(t)=0$. We further make the following assumption on the discount function $\delta(\cdot)$.

\begin{Assumption}\label{assume.delta}
$\delta(\cdot)$ is log sub-additive, i.e.,
\be\label{eq.assume.logsubadd} 
\delta(t+s)\geq \delta(t)\delta(s),\quad \forall s,t\geq 0.
\ee 
\end{Assumption}
\begin{Remark}
Typical discount functions, including exponential, hyperbolic, generalized hyperbolic and pseudo-exponential discounting, satisfy Assumption \ref{assume.delta}.
\end{Remark}

Given the transition kernel $Q(x,dy)$ for $X$ and a stopping time $\tau$, define 
$$
v^Q(x,\tau,f):=\E_x^Q[\delta(\tau)f(X_\tau)],
$$
where $\E_x^Q$ is the expectation w.r.t. $Q$ given $X_0=x$. For $S\in\Bc$, denote
 $$\rho(S):= \inf\{t\geq 1, X_t\in S\},$$
 and
$$J^Q(x,S,f):=\E^Q_x[\delta(\rho(S))f(X_{\rho(S)})]\cdot 1_{\{x\notin S\}}+f(x)\cdot 1_{\{x\in S\}},\quad \forall x\in \X.$$
We provide the definition of equilibria and optimal equilibria in the following.

\begin{Definition}[Equilibria and optimal equilibria]\label{def.equilibrium}
Fix a reward function $f$ and a transition kernel $Q$. A Borel set $S\subset\X$ is called an equilibrium (w.r.t. $f$ and $Q$) if	
	\be\label{eq.def.equilibrium}  
\begin{cases}
f(x)\leq \E^Q_x[\delta(\rho(S))f(X_{\rho(S)})],\quad \forall x\notin S,\\
 f(x)\geq \E^Q_x[\delta(\rho(S))f(X_{\rho(S)})],\quad \forall x\in S.
\end{cases}
\ee 
Denote $\Ec^Q(f)$ the set of equilibria w.r.t. $f$ and $Q$.
$S\in\Ec^Q(f)$ is called an optimal equilibrium (w.r.t. $f$ and $Q$), if for any $T\in \Ec^Q(f)$,
$$
J^Q(x, S,f)\geq J^Q(x,T,f),\quad \forall x\in \X.
$$
\end{Definition}

Let
\be\label{eq.value.optima} 
V^Q(x,f):=\sup_{S\in \Ec^Q(f)} J^Q(x,S,f),\quad x\in \X,
\ee
%and
%\begin{equation}\label{e001}
%\quad S^*(f,Q):= \cap_{S\in \Ec^Q(f)} S.
%\end{equation}
which represents the optimal value generated over all equilibria. As indicated by results in \cite{MR3911711} (also see Lemma \ref{lm.iteration.sstar}) there exists an optimal equilibria and thus the supremum for $V^Q(x,f)$ is attained universally at the optimal equilibria for all $x\in\X$.
%\begin{Lemma}\label{lm.optimalmild}
%Suppose the transition kernel $Q(x,.)$ is lower semi-continuous under the weak star topology, and reward function $f$ is upper semi-continuous (or lower semi-continuous), then $S^*(f,Q)$ is the smallest optimal mild equilibrium.
%\end{Lemma}
%
%\begin{Remark}
%When $\X$ is discrete, the reward function can be arbitrary non-negative and bounded function, since it is continuous w.r.t. the discrete topology.
%\end{Remark}
%
%Through this paper, to garantee the optimality of $S^*(f,Q)$, we assume that all the reward functions that considered in this paper are either upper semi-continuous or lower semi-continuous, and the transition kernels in this paper are all lower semi-continuous under the weak star topology.
In this paper, we investigate the stability of $V^Q(x,f)$ w.r.t. the transition kernel $Q$ and reward function $f$. To begin with, recall the total variation distance between two measures $\mu$ and $\nu$,
$$ ||\mu- \nu||_\text{TV} := \sup_{g\in B(\X;[0,1])} \left \{ \int_\X g \, d\mu - \int_\X g \, d\nu \right \},$$
where $B(\X;[0,1])$ is the set of Borel measurable functions on $\X$ taking values in $[-1,1]$. We will use the following notions of convergence for $f$ and $Q$ for the stability analysis of $V^Q(x,f)$.

\begin{Definition}\label{def.f}
	Let $(f^n)_{n\in \overline \N}$ be a sequence of functions on $\X$. We say $f^n$ converges to $f^\infty$ locally uniformly if for any compact set $K\subset \X$,
$$\lim_{n\to\infty}\sup_{x\in K} |f^n(x)-f^\infty(x)|=0.$$
	Recall that $f^n$ converges to $f^\infty$ uniformly if $\|f^n-f^\infty\|_{\infty}\to 0$ as $n\to\infty$.
\end{Definition}

\begin{Definition}\label{def.variation}
	Let $(Q^n)_{n\in \overline \N}$ be a sequence of transition kernels. We say $Q^n$ converges to $Q^\infty$ locally uniformly in total variation, if for any compact set $K\subset\X$,
	$$\lim_{n\to\infty} \sup_{x\in K}||Q^n(x,\cdot)-Q^\infty(x,\cdot)||_\text{TV}= 0.$$
	We say $Q^n$ converges to $Q^\infty$ uniformly in total variation, if
	$$\lim_{n\to\infty}\sup_{x\in \X}||Q^n(x,\cdot)-Q^\infty(x,\cdot)||_\text{TV}=0.$$
\end{Definition}

\begin{Remark}
	When $\X$ is countable and under the discrete topology, locally uniform convergence of $(Q^n(x,y))_{n\in\overline \N}$ in total variation is the same as the pointwise weak convergence.
	
	When $\X$ is uncountable (e.g., the process under $Q^n$ is a time-discretized diffusion), then the locally uniform convergence of $(Q^n(x,y))_{n\in\overline \N}$ in total variation can be implied by the following condition: There exist a reference measure $\mu$ such that any $Q^n(x,\cdot)$ has a probability density $q^n(x,\cdot)$ w.r.t. $\mu$, i.e., $Q^n(x,dy)=q^n(x,y)\mu(dy)$, and for any compact set $K\subset\X$,
	$$
	\lim_{n\to\infty}\sup_{x\in K}\int_\X|q^n(x,y)-q^\infty(x,y)|d\mu(y)=0.
	$$	
\end{Remark}

Now we present three lemmata that will be used in later sections, and their proofs are collected in Appendix \ref{appendix}.
The first lemma is an analogue of Theorem 2.2 in \cite{MR4205889} for discrete time setting, which provides the existence of an optimal equilibrium, as well as an iterative approach for its construction. To this end, define
\begin{equation}\label{e001}
\quad S^*(f,Q):= \cap_{S\in \Ec^Q(f)} S.
\end{equation}
We have the following.
\begin{Lemma}\label{lm.iteration.sstar}
	Let Assumption \ref{assume.delta} hold. Suppose $f$ is bounded and non-negative, and Q is a transition kernel. Define $S_0=\emptyset$ and for $k=1,2,\dotso$,
	$$S_{k+1}:=S_k\cup\left\{x\in\X\setminus S_k:\ f(x)>\sup_{1\leq\tau\leq\rho(S_k)}v^Q(x,\tau,f)\right\}.$$
	Then $\cup_{k\in \N} S_k=S^*(f,Q)$. Moreover, $S^*(f,Q)$ is an optimal equilibrium, and thus
	$$
	V^Q(x,f)=J^Q(x, S^*(f,Q),f),\quad \forall x\in \X.
	$$
\end{Lemma}

\begin{Remark}
Lemma \ref{lm.iteration.sstar} indicates that there exists a ``smallest'' equilibrium, which is also an optimal one. The supremum for $V^Q(x,f)$ is achieved by the same equilibrium $S^*(f,Q)$. Moreover, If the discount function is exponential, i.e., when the stopping problem \eqref{eq.intro.optimalstopping} is time-consistent, a similar discussion as that in \cite{MR4205889} would show that $S^*(f,Q)$ and $V^Q(x,f)$ would coincide with the optimal stopping region and value respectively in the classical sense.
\end{Remark}

\begin{Lemma}\label{l2}
		Let $(Q^n)_{n\in \overline \N}$ be transition kernels.
	\bi
	\item[(a)] Suppose $Q^n$ converges to $Q^\infty$ locally uniformly in total variation. Then for any $x\in\X$ and $T\in\N$,
	$$\lim_{n\to\infty}\sup_{g\in B(\X^T;[0,1])} \left |\E_x^{Q^n}g(X_1,X_2,\dotso,X_T)-\E_x^{Q^\infty} g(X_1,X_2,\dotso,X_T)\right |=0.$$
	\item[(b)]
	In addition to the condition in part (a), assume that, for any compact set $K$ and $\eps>0$, there exists a compact set $K'$ such that $\sup_{x\in K} Q^\infty(x, K')\geq 1-\varepsilon$.
	Then for any compact set $K$ and $T\in\N$,
	$$\lim_{n\to\infty}\sup_{x\in K, g\in B(\X^T;[0,1])} \left |\E_x^{Q^n}g(X_1,X_2,\dotso,X_T)-\E_x^{Q^\infty} g(X_1,X_2,\dotso,X_T)\right |=0.$$
	\item[(c)] Suppose $Q^n$ converges to $Q^\infty$ uniformly in total variation.
	Then for any $T\in\N$,
	$$\lim_{n\to\infty}\sup_{x\in \X, g\in B(\X^T;[0,1])} \left |\E_x^{Q^n}g(X_1,X_2,\dotso,X_T)-\E_x^{Q^\infty} g(X_1,X_2,\dotso,X_T)\right |=0.$$
	\ei
\end{Lemma}

\begin{Remark}
Suppose under $Q^\infty$,
$$X_{t+1}=h(X_t,\xi_t),$$
where $\xi_0,\xi_1,\dotso$ are i.i.d. random variables and $h:\X\times\R^d\mapsto\X$ is continuous. Then the additional assumption in Lemma \ref{l2}(b) is satisfied. Indeed, fix compact set $K\subset\X$ and $\eps>0$. There exists constant $C>0$ such that $\P(|\xi_0|\leq C)\geq 1-\eps$. Let $C':=\sup_{(x,y)\in K\times \overline{B_{C}}}|h(x,y)|<\infty$ and $K':=\overline{B_{C'}}\subset\X$, where $B_r$ is the ball centered at zero with radius $r$. Then $\sup_{x\in K} Q^\infty(x,K')\geq\P(|\xi_0|\leq C)\geq 1-\eps$.
\end{Remark}

%
%\begin{corollary}\label{c1}
%Assume $Q^n$ converges to $Q^\infty$ locally uniformly in total variation. Let $\T'\subset\T$ be a class of stopping times. Then
%$$\sup_{\tau\in\T'}v^{Q^n}(x,\tau)\to \sup_{\tau\in\T'}v^\infty(x,\tau).$$
%\end{corollary}

\begin{Lemma}\label{lm.tau.uniform}
	Let $(Q^n)_{n\in \overline \N}$ be transition kernels, and $(f^n)_{m\in \overline \N}$ be non-negative reward functions such that $\sup_{n\in \overline \N} \|f^n\|_{\infty}<\infty$. Suppose Assumption \ref{assume.delta} holds.
	\bi 
	\item[(a)] Suppose $Q^n$ converges to $Q^\infty$ locally uniformly in total variation and $f^n$ converges to $f^\infty$ locally uniformly. Then
	$$
	\lim_{n\rightarrow\infty} \sup_{\tau\in \T} |v^{Q^n}(x,\tau,f^n)-v^{Q^\infty}(x,\tau,f^\infty)|=0,\quad \forall x\in \X.
	$$
	\item[(b)] In addition to the conditions in part (a), assume that
	for any compact set $K$ and $\eps>0$, there exists a compact set $K'$ such that $\sup_{x\in K} Q^\infty(x, K')\geq 1-\varepsilon$.
	Then for any compact set $K$,
	$$
	\lim_{n\rightarrow\infty} \sup_{x\in K, \tau\in \T} |v^{Q^n}(x,\tau,f^n)-v^{Q^\infty}(x,\tau,f^\infty)|=0.
	$$
	\ei 
	\item[(c)]
	Suppose $Q^n$ converges to $Q^\infty$ uniformly in total variation
	and $\|f^n-f^\infty\|_{\infty}\rightarrow0$. 
	Then
	$$
	\lim_{n\rightarrow\infty} \sup_{x\in \X,\tau\in \T} |v^{Q^n}(x,\tau,f^n)-v^{Q^\infty}(x,\tau,f^\infty)|=0.
	$$
	
\end{Lemma}

%For $n\in\overline\N$, let $\hat S^n$ be the smallest optimal equilibria w.r.t. $Q^n$, and $V^n(x):=f(x)\vee v(x,\rho(\hat S^n))$ be the associated value.
%\begin{equation}\label{e1}
%\lim_{L\to\infty}\sup_{n\in\overline\N}\int_{f(y)>L}f(y)Q^n(x,dy)\to 0\quad\text{and}\quad\sup_{n\in\overline\N}\lim_{t\to\infty}\E_x^n\left[\sup_{s\geq t}\delta(s)f(X_s)\right]=0.
%\end{equation}

%Consider the time-inconsistent stopping problem
%$$v^n(x,\tau):=\E_x^n[\delta(\tau)f(X_\tau)].$$	
	
\section{Semi-contintuity of the smallest optimal equilibrium and its associated value}\label{sec:upper.continue}	

%\subsection{The first main result}\label{subse:.first.result}
In this section, we present the first main result: the semi-continuity of $V^Q(x,f)$ and $S^*(f^\infty, Q^\infty)$ w.r.t. $f$ and $Q$. The proof is collected in Section \ref{subsec:thm1.proof}. Examples for discontinuity are also provided.

\begin{Theorem}\label{t1}
Suppose Assumption \ref{assume.delta} holds. Let $(Q^n)_{n\in \overline \N}$ be transition kernels, and $(f^n)_{n\in \overline \N}$ be non-negative reward functions with $\sup_{n\in \overline \N} \|f^n\|_{\infty}<\infty$. Suppose $Q^n$ converges to $Q^\infty$ locally uniformly in total variation, and $f^n$ converges to $f^\infty$ locally uniformly. Then
\begin{equation}\label{e5}
S^*(f^\infty, Q^\infty) \subset \liminf_{n\rightarrow\infty}S^*(f^n,Q^n),
\end{equation}
and
\begin{equation}\label{e6}
V^{Q^\infty}(x,f^\infty)\geq\limsup_{n\rightarrow\infty}V^{Q^n}(x,f^n),\quad \forall x\in \X. 
%\footnote{Notice that both the limits in \eqref{e5} and \eqref{e6} are doubled limits. And taking intersection and union of a two dimensional indiced sequence of sets are always well-defined, so the RHS of \eqref{e6} exists.}
\end{equation}
\end{Theorem}

\begin{Remark}\label{rm.sets.semicontinue}
We also have the semi-continuity in terms of the equilibria sets: under the conditions in Theorem \ref{t1}, 
$$
\limsup_{n\to \infty} \Ec^{Q^n}(f^n)\subset  \Ec^{Q^\infty}(f^\infty).
$$
Indeed, for $S\in \limsup_{n\to \infty} \Ec^{Q^n}(f^n)$, there exists a subsequence $(n_k)_k$ such that $S\in \Ec^{Q^{n_k}}(f^{n_k})$, and thus
$$
\begin{cases}
	f^{n_k}(x)\leq \E^{Q^{n_k}}[\delta(\rho(S)f^{n_k}(X_{\rho(S)}))],\quad \forall x\notin S;\\
	f^{n_k}(x)\geq \E^{Q^{n_k}}[\delta(\rho(S)f^{n_k}(X_{\rho(S)}))],\quad \forall x\in S.
\end{cases}
$$
By Lemma \ref{lm.tau.uniform}(a), letting $k\to\infty$ we can conclude that $S\in \Ec^{Q^\infty}(f^\infty)$.
\end{Remark}

\begin{Remark}\label{rm.exponential}
	If $\delta$ is exponential, i.e., $\delta(t+s)=\delta(t)\delta(s)$ for any $s,t\geq 0$, then by a similar discussion as that in \cite{MR4205889}, we have that 
	\be\label{eq.expon0} 
	V^{Q^n}(x,f^n)= \sup_{\tau\in \T} \E^{Q^n}_x[\delta(\tau) f^n(X_\tau)].
	\ee
By Lemma \ref{lm.tau.uniform}(a), \eqref{eq.expon0} implies that
	\be\label{eq.expon} 
	\lim_{n\to\infty} V^{Q^n}(x,f^n)=V^{Q^\infty}(x,f^\infty),\quad \forall x\in \X,
	\ee
	which is the continuity of the optimal value function. 
%Then \eqref{eq.expon} also leads to the convergence of the ``largest" optimal stopping regions: 
%	$$
%	\{x: V^{Q^n}(x,f^n)=f^n(x)\}\overset{n\to\infty}{\longrightarrow}\{x: V^{Q^\infty}(x,f^\infty)=f^\infty(x)\}.
%	$$ 
%Such stability result is already known, for e.g. \cite{.}. {\bf no good reference in discrete time found, only continuous time reference for stability of process law}
However, 	we still only have the semi-continuity for the ``smallest" optimal stopping region $S^*(f^n,Q^n)$. 
%the semi-continuity for the ``smallest" optimal equilibrium:
%	$$
%S^*(f^\infty, Q^\infty) \subset \liminf_{n\rightarrow\infty}S^*(f^n,Q^n).
%	$$
\end{Remark}

We now present three examples of discontinuity. The first two examples shows that the strict inequalities in \eqref{e5} and \eqref{e6} can happen. Example \ref{eg.upper.finite} is for discontinuity w.r.t. the transition kernel, and Example \ref{eg.upper.finite1} is for discontinuity w.r.t. the reward function. Then we provide a discontinuity example under weak convergence of transition kernels, which indicates that the locally uniform convergence in total variation for transition kernels is the right assumption. 
\begin{Example}\label{eg.upper.finite}
	Let $\X=\{a,b,c\}\subset \R$ with $c<b<a$, $\delta(1)=1/2$ and $\delta(2)=1/3$. Define
	\begin{align*}
		\begin{cases}
			Q^n: Q^n(c,b)=1, \quad Q^n(b,a)=p_n=1-\frac{1}{n}, \quad Q^n(b,b)=\frac{1}{n}, \quad \forall n\in \overline \N,\\
			f(a)=2, \quad  f(b)=1, \quad  f(c)=\frac{1}{2},
		\end{cases}
	\end{align*}
where $\frac{1}{\infty}:=0$, and with a bit of abuse of notation $Q(x,y):=\P(X_1=y|X_0=x)$.  It is easy to check that $Q^n$ converges to $Q^\infty$ uniformly in total variation.
	
	Note that any equilibrium must contain the global maximum of the reward function. By computation 
	$$
	J^{Q^\infty} (b, \{a\},f)=1=f(b)\quad\text{and}\quad J^{Q^\infty}(c,\{a\},f)=2/3>f(c),
	$$
	which imply that $S^*(f,Q^\infty)=\{a\}$. Moreover, since for $n<\infty$,
	$$f(b)>J^{Q^n}(b,\{a\},f)=J^{Q^n}(b,\{a,c\},f),$$
	any equilibrium w.r.t. $Q^n$ for $n<\infty$ must contain $\{a,b\}$.  As $f(c)>\delta(1)f(b)$, $\{a,b\}$ is not  equilibrium w.r.t. $Q^n$ for $n<\infty$. Consequently, $\Ec^{Q^n}(f)=\{\X\}$ for $n<\infty$. Hence, 
	$$
	S^*(f,Q^\infty)=\{c\}\subsetneqq \X=S^*(f,Q^n),\quad \forall n<\infty,$$
	and
	$$V^{Q^n}(c,f)=f(c)<J^{Q^\infty}(c,\{a\})=V^{Q^\infty}(c,f).$$
\end{Example}

\begin{Example}\label{eg.upper.finite1}
	Let $\X=\{a,b,c\}\subset \R$ with $c<b<a$, $\delta(1)=1/2$ and $\delta(2)=1/3$. Define
	\begin{align*}
		\begin{cases}
			Q(c,b)=1, \quad Q(b,a)=1, \quad Q(a,a)=1,\\
			f^n(a)=2, \quad f^n(b)=1+\frac{1}{n},\quad f^n(c)=\frac{1}{2}+(1+\delta(1))\frac{1}{n},\quad \forall n\in \overline \N.
		\end{cases}
	\end{align*}
	Obviously, $\|f^n-f^\infty\|_{\infty}\rightarrow\infty$.
	
	We can compute that
	$$
	J^Q (b, \{a\}, f^\infty)=1=f^\infty(b), \quad J^Q(c,\{a\},f^\infty)=2/3>f^\infty(c),$$
	and thus $\hat{S}^\infty=\{a\}$. Meanwhile,
	$$J^Q(b,\{a\},f^n)=J^Q(b,\{a,c\},f^n)=1<f^n(b),$$
	so neither $\{a\}$ nor $\{a,c\}$ belongs to $\Ec^Q(f^n)$ for $n<\infty$. By $$
	f^n(c)=\frac{1}{2}+(1+\delta(1))\frac{1}{n}>\frac{1}{2}+\delta(1)\frac{1}{n}=\delta(1)f^n(b),
	$$ 
	$\{a,b\}$ is not  equilibrium for all $f^n$ for $n<\infty$. Therefore, $\X$ is the only  equilibrium w.r.t. $f^n$ for $n<\infty$. Hence,
	$$S^*(f^\infty,Q)=\{c\}\subsetneqq \X=S^*(f,Q^n),\quad\forall\,n<\infty,$$
	and
	$$
	\limsup_{n\rightarrow \infty}V^Q(c,f^n)=\limsup_{n\rightarrow \infty}f^n(c)=\frac{1}{2}<\frac{2}{3}=V^Q(c,f^\infty).
	$$
\end{Example}

%\subsection{Discontinuity under weak convergence of $(Q^n)_{n\in \overline \N}$}

When $\X$ is finite, convergence locally uniformly in total variation is equivalent to weak convergence. When $\X$ is not finite, we provide below an example showing that the semi-continuity in Theorem \ref{t1} fails when only weak convergence is assumed. Hence, weak convergence is too weak to establish the semi-continuity in Theorem \ref{t1}. %Hence, convergence locally uniformly in total variation of $Q^n$ to $Q^\infty$ is the right condition to guarantee the semi-continuity.

\begin{Example}\label{eg.weakconvergence}
	Let $\X=\{y, x_\infty,x_1,x_2,...\}\subset \R$, where $0\leq x_n\nearrow x_\infty$ and $y=\dfrac{x_\infty}{\delta(2)}+1$. Let $f(x)=x$. Define for $n<\infty$,
	\begin{align*}
		Q^n:& 
		\begin{cases}
			Q^n(x_i, x_n)=1, \quad \text{for} \quad i\neq n,\\
			Q^n(x_\infty,x_n)=1, Q^n(x_n,y)=1, Q^n(y,y)=1,
		\end{cases}
	\text{and} \quad 
		Q^\infty:&
		\begin{cases} Q^\infty(x_i, x_\infty)=1, \quad \text{for} \quad \forall i,\\
			Q^\infty(x_\infty, x_\infty)=1, Q^\infty(y, y)=1. 
		\end{cases}
	\end{align*}
	It can be shown that $Q^n(z,\cdot)$ weakly converges to $Q^\infty(z,\cdot)$ for any $z\in\X$.
%	Take any continuous function $g$,
%	$\lim_{n\rightarrow \infty}\E^{Q^n}_{x_i}(g(X_1))=\lim_{n\rightarrow \infty}g(x_n)=g(x_\infty)$, for $i\in \overline \N$; $ \E^{Q^n}_y[g(X_1)]=g(y)$, $\forall n\in \overline \N$.
However, since $Q^n(x_1,\{x_\infty\})=0$ for $n<\infty$ while $Q^\infty(x_1,\{x_\infty\})=1$, the locally uniform convergence in total variation fails.
	
	For $n<\infty$, since $y>\dfrac{x_\infty}{\delta(2)}$, we have that
	$$
	\E^{Q^n}_{x_i}[\delta(\rho(\{y\})f(X_{\rho(\{y\}}))]=\begin{cases}
	\delta(2)y, & i\in \overline \N\setminus\{n\}\\
	\delta(1)y,&i=n
	\end{cases}\ >x_\infty\geq x_i.
	$$ 
	%Notice that $\E^{Q^n}_{y}[\delta(\rho(\{y\})f(X_{\rho(\{y\}}))]=\delta(1)y<y$. 
This implies $S^*(f,Q^n)=\{y\}$ for $n<\infty$. On the other hand, denote
$$S_1:=\left\{x\in\X:\ f(x)>\sup_{1\leq\tau}v^{Q^\infty}(x,\tau,f)\right\}.$$
Obviously, $\{x,y\}\subset S_1$. By Lemma \ref{lm.iteration.sstar}, we have that $\{x,y\}\subset S^*(f,Q^\infty)$. Hence,
	$$
	 \limsup_{n\rightarrow \infty} S^*(f,Q^n)\subsetneqq S^*(f,Q^\infty) \quad \text{and}\quad V^{Q^\infty}(x_\infty,f)=x_\infty<\liminf_{n\to \infty}V^{Q^\infty}(x_\infty,f)=\delta(2)y.
	$$ 
\end{Example}

\subsection{Proof of Theorem \ref{t1}}\label{subsec:thm1.proof}

\begin{proof}[{\bf \textit{Proof of Theorem \ref{t1}}}]
	For $n\in\overline\N$, define $S_0^n=\emptyset$ and
	\be\label{eq.thm.0} 
	S_{k+1}^{n}:=S_k^n\cup\left\{x\in\X\setminus S_k^{n}:\ f(x)>\sup_{1\leq\tau\leq\rho(S_k^{n})}v^{Q^n}(x,\tau,f^n)\right\}.
	\ee
	By Lemma \ref{lm.iteration.sstar}, $S^*(f^n,Q^n)=\cup_k S_k^{n}=\lim_{k\to\infty}S_k^{n}$, $\forall n\in \overline \N$. We show by induction that
	\begin{equation}\label{e3}
		S_k^{\infty}\subset \liminf_{n\rightarrow\infty} S_k^{n},\quad k=0,1,\dotso,
	\end{equation}
	which in particular implies that $S^*(f^\infty,Q^\infty)\subset\liminf_{n\to\infty}S^*(f^n,Q^n)$.
	
	Obviously, \eqref{e3} holds for $k=0$. Suppose it holds for $k=i$ and consider the case $k=i+1$. Take $x\in S_{i+1}^{\infty}$. If $x\in S_i^{\infty}$, then by induction hypothesis $$
	x\in\liminf_{n\to\infty} S_{i}^{n}\subset \liminf_{n\to\infty} S_{i+1}^{n}.
	$$
	Now assume $x\notin S_i^{\infty}$. Then
	\begin{equation}\label{e4}
		\alpha:=f^\infty(x)-\sup_{1\leq\tau\leq\rho(S_{i}^{\infty})}v^{Q^\infty}(x,\tau,f^\infty)>0.
	\end{equation}
	Denote the probability measure $\P^n$ induced by $Q^n$. By induction hypothesis,
	$$\rho({S_i^{\infty}})\geq\rho\left(\underset{1\leq n<\infty}{\cup}\left(\underset{n\leq j<\infty,}{\cap}S_i^{j}\right)\right)=\lim_{n\to\infty}\rho\left(\underset{n\leq j<\infty}{\cap}S_i^{j}\right),\quad\P_x^\infty-\text{a.s.}.$$
	Therefore, there exists $N\in\N$ such that for any $n\geq N$,
\be\label{e7}
		\P_x^{Q^\infty}\left[\rho(S_i^{n})>\rho({S_i^{\infty}})\right]\leq\P_x^{Q^\infty}\left[\rho\left(\underset{n\leq j<\infty}{\cap}S_i^{j}\right)>\rho({S_i^{\infty}})\right]<\frac{\alpha}{2M},
\ee 
where $M:=\sup_{n\in \overline\N} \|f^n\|_{\infty}<\infty$.
	Then for any $\tau'$ with $1\leq\tau'\leq\rho(S_i^{n})$, we have that
	$$v^{Q^\infty}(x,\tau',f^\infty)\leq v^{Q^\infty}(x,\tau'\wedge \rho(S_i^{\infty}),f^\infty)+\frac{\alpha}{2}\leq \sup_{1\leq\tau\leq\rho(S_i^{\infty})}v^{Q^\infty}(x,\tau,f^\infty)+\frac{\alpha}{2},$$
	and thus
	$$\sup_{1\leq\tau\leq\rho(S_i^{n})}v^{Q^\infty}(x,\tau,f^\infty)\leq  \sup_{1\leq\tau\leq\rho(S_i^{\infty})}v^{Q^\infty}(x,\tau,f^\infty)+\frac{\alpha}{2},\quad\forall\,n\geq N.$$
	This together with \eqref{e4} implies that
	\be\label{eq.1.thm} 
	f^\infty(x)-\sup_{1\leq\tau\leq\rho(S_{i}^{n})}v^{Q^\infty}(x,\tau,f^\infty)\geq\frac{\alpha}{2}>0.\ee
	By Lemma \ref{lm.tau.uniform} part (a), for $n$ large enough, we have that
	\be\label{eq.2.thm} 
	\begin{aligned}
	\left|\sup_{1\leq\tau\leq\rho(S_{i}^{n})}v^{Q^\infty}(x,\tau,f^\infty)-\sup_{1\leq\tau\leq\rho(S_{i}^{n})}v^{Q^n}(x,\tau,f^n)\right|\leq & \sup_{1\leq\tau\leq\rho(S_{i}^{n})}\left|v^{Q^\infty}(x,\tau,f^\infty)-v^{Q^n}(x,\tau,f^n)\right| \\
	\leq & \sup_{\tau\in\T}\left|v^{Q^\infty}(x,\tau,f^\infty)-v^{Q^n}(x,\tau,f^n)\right|<\frac{\alpha}{3}.
	\end{aligned}
\ee
Meanwhile, we can choose $N'$ such that for all $n\geq N'$ \eqref{eq.2.thm} holds and 
\be\label{eq.3.thm}  
|f^n(x)-f^\infty(x)|\leq \frac{\alpha}{12}.
\ee 
Thus, for all $n\geq \max\{N, N'\}$, combine \eqref{eq.1.thm}, \eqref{eq.2.thm} and \eqref{eq.3.thm},
\begin{align*}
f^n(x)-	\sup_{1\leq\tau\leq\rho(S_{i}^{n})}v^{Q^n}(x,\tau,f^n)=& f^n(x)-f^\infty(x)+f^\infty(x)-\sup_{1\leq\tau\leq\rho(S_{i}^{n})}v^{Q^\infty}(x,\tau,f^\infty)\\
&+\sup_{1\leq\tau\leq\rho(S_{i}^{n})}v^{Q^\infty}(x,\tau,f^\infty)-\sup_{1\leq\tau\leq\rho(S_{i}^{n})}v^{Q^n}(x,\tau,f^n)\\
\geq&-\frac{\alpha}{12}+\frac{\alpha}{2}-\frac{\alpha}{3}>0.
\end{align*}	
	Consequently, for $n$ large enough, no matter $x$ is in $S_i^{n}$ or not, we always have $x\in S_{i+1}^{n}$, and thus $x\in\liminf_{n\to\infty} S_{i+1}^{n}$. By the arbitrariness of $x$, \eqref{e3} holds for $k=i+1$. We have proved \eqref{e5}.
	
	Now let $\eps>0$ and $x\notin S^*(f^\infty,Q^\infty)$. Following the argument in \eqref{e7}, we can show that there exists $N\in\N$ such that for any $n>N$,
	\begin{equation}\label{e8}
		\P_x^{Q^\infty}\left[\rho(S^*(f^n,Q^n))>\rho({S^*(f^\infty,Q^\infty)})\right]<\frac{\eps}{2M}.
	\end{equation}
	Then there exists $N'>N$ such that for any $n>N'$,
	\begin{align*}
	v^{Q^\infty}(x,\rho(S^*(f^\infty,Q^\infty)))\geq &v^{Q^\infty}(x,\rho(S^*(f^\infty,Q^\infty)\cup S^*(f^n,Q^n)))\geq v^{Q^\infty}(x,\rho(S^*(f^n,Q^n)))-\frac{\eps}{2}\\
	\geq & v^{Q^n}(x,\rho(S^*(f^n,Q^n)))-\eps,
	\end{align*}
	where the first inequality follows from \cite[Lemma 3.1]{MR4116459} (or Lemma \ref{lm.equi.pseudo}), the second inequality follows from \eqref{e8}, the third inequality follows from Lemma \ref{lm.tau.uniform} part (a). As a result,
	$$v^{Q^\infty}(x,\rho(S^*(f^\infty,Q^\infty)))\geq \limsup_{n\to\infty}v^{Q^\infty}(x,\rho(S^*(f^n,Q^n)))-\eps.$$
	By the arbitrariness of $\eps$, we have \eqref{e6} holds.
\end{proof}	
	
\section{Continuity under a relaxed limit}\label{sec:continue.epsilon}

%\section{$a\varepsilon$-(pseudo)  equilibri}\label{sec:epsilon,pseudo}	

As shown in the previous section, $V^Q(x,f)$ is not continuous w.r.t. $Q$ or $f$ in general. To achieve the stability, we need to relax the equilibrium set over which we take supremum. 
	
\begin{Definition}\label{def.equi.epsi}
	Fix a reward function $f$ and a transition kernel $Q$. Take $\varepsilon\geq 0$. A Borel set $S$ is called an $\varepsilon$-equilibrium (w.r.t. $f$ and $Q$), if 
	\be \label{eq.def.epsequi}
\begin{cases}
	f(x)\leq \E^Q_x[\delta(\rho(S))f(X_{\rho(S)})]+\eps,\quad \forall x\notin S,\\
	f(x)+\eps\geq \E^Q_x[\delta(\rho(S))f(X_{\rho(S)})],\quad \forall x\in S.
\end{cases}
	\ee 
	Define
	$$
	\Ec^Q(f,\varepsilon):=\{\text{$S$ is an $\varepsilon$-equilibrium w.r.t. }f \text{ and }Q\}.
	$$
	When $\varepsilon=0$, we still call $S$ an equilibrium and may use the notation $\Ec^Q(f)$ instead of $\Ec^Q(f,0)$.
\end{Definition}	

We also need the following notion of pseudo $\varepsilon$-equilibria, which loosens the criterion of $\eps$-equilibrium by giving up the condition in \eqref{eq.def.epsequi} when $x\in S$.

\begin{Definition}\label{def.equi.pseudoepsi}
Fix a reward function $f$ and a transition kernel $Q$. Take $\varepsilon\geq 0$. A Borel set $S\subset\X$ is called a pseudo $\varepsilon$-equilibrium (w.r.t. $f$ and $Q$), if 
	\be\label{eq.defem.outS}  
 f(x)\leq \E^Q_x[\delta(\rho(S))f(X_{\rho(S)})]+\varepsilon,\quad \forall x\notin S.
\ee 
Define
$$
\Gc^Q(f,\varepsilon):=\{\text{$S$ is a pseudo $\varepsilon$-equilibrium w.r.t. $f$ and $Q$}\}.
$$
When $\varepsilon=0$, we simply call $S$ is a pseudo  equilibrium, and write $\Gc^Q(f)$ short for $\Gc^Q(f,0)$. We say $S\in \Gc^Q(f)$ is an optimal pseudo equilibrium (w.r.t. $f$ and $Q$), if for any $T\in \Gc^Q(f)$,
$$
J(x,S,f)\geq J(x,T,f),\quad \forall x\in \X.
$$
\end{Definition}	
%
%Compare with Definition \ref{def.equilibrium}, Definition \ref{def.equi.epsi} is  indeed a generalized version of true equilibrium for a $\varepsilon$-level. 

Now define
\be\label{eq.def.WVepsilon} 
W^Q_\varepsilon(x,f):=\sup_{S\in \Gc^Q(f,\varepsilon)} J^Q(x,S,f);\quad V^Q_\varepsilon(x,f):=\sup_{S\in \Ec^Q(f,\varepsilon)} J^Q(x,S,f).
\ee
When $\varepsilon=0$ we write $W^Q(x,f)$ instead of $W^Q_0(x,f)$, and we keep using the notation $V^Q(x,f)$ in \eqref{eq.value.optima} instead of $V^Q_0(x,f)$. 

Pseudo $\eps$-equilibria have better properties than $\eps$-equilibria. As we will see in Lemma~\ref{lm.equi.epsilon} below one can embed the set of pseudo-$\eps$-equilibria to pseudo equilibria corresponding to a perturbed reward function. We will also observe that the smallest optimal pseudo equilibrium is actually the smallest optimal equilibrium in Proposition~\ref{prop.optimalequi.pseudo}. These two results form the backbone of the proof of the second main result which we state below. The proof of this result is provided in Section \ref{subsec:proof.continue}.

\begin{Theorem}\label{thm.Qf.continuity}
	Suppose Assumption \ref{assume.delta} holds. Let $(Q^n)_{n\in \overline \N}$ be transition kernels, and $(f^n)_{n\in \overline \N}$ be bounded and non-negative reward functions. 
	Suppose $Q^n$ converges to $Q^\infty$ uniformly in total variation, and $\|f^n-f^\infty\|_{\infty}\rightarrow0$. Then
	\begin{align*}
		&\lim\limits_{\varepsilon\searrow 0}\Big( \liminf_{n\rightarrow \infty} V^{Q^n}_\varepsilon(x,f^n)\Big)=  \lim\limits_{\varepsilon\searrow 0}\Big( \liminf_{n\rightarrow \infty} W^{Q^n}_\varepsilon(x,f^n)\Big)\\
		=&\lim\limits_{\varepsilon\searrow 0}\Big( \limsup_{n\rightarrow \infty}V^{Q^n}_\varepsilon(x,f^n)\Big)=  \lim\limits_{\varepsilon\searrow 0}\Big( \limsup_{n\rightarrow \infty} W^{Q^n}_\varepsilon(x,f^n)\Big)\\
		=&V^{Q^\infty}(x,f^\infty),\quad \forall x\in \X.
	\end{align*}	
\end{Theorem}

Letting $f^n=f$ and $Q^n=Q$ for $n\in \overline \N$ in Theorem \ref{thm.Qf.continuity}, we achieve the following corollary, which shows that $V^Q(x,f)$ is indeed the limit of the supremum value over all $\varepsilon$-equilibria as $\varepsilon\searrow 0$.

\begin{Corollary}\label{cor}
	Suppose Assumption \ref{assume.delta} holds. Given a bounded reward function $f\geq 0$ and a transition kernel $Q$, we have that
	$$
	\lim_{\varepsilon\searrow 0}V^Q_\varepsilon(x,f)=\lim_{\varepsilon\searrow 0}W^Q_\varepsilon(x,f)=V^Q(x,f),\quad \forall x\in \X.
	$$
\end{Corollary}

\begin{Remark}\label{rm.exchange.application}
	Combining Theorem \ref{t1} and Corollary \ref{cor}, we have
	$$
	\limsup\limits_{n\rightarrow \infty}\Big( \lim_{\varepsilon\searrow 0}V^{Q^n}_\varepsilon(x,f^n)\Big)=\limsup\limits_{n\rightarrow \infty} V^{Q^n}(x,f^n)\leq V^{Q^\infty}(x,f^\infty),\quad \forall x\in\X.
	$$
	Recall that the strict inequality above can be achieved as shown in Examples \ref{eg.upper.finite} and \ref{eg.upper.finite1}. Hence, together with Theorem \ref{thm.Qf.continuity}, we see that
	the order of taking $\varepsilon\searrow 0$ and taking $n\rightarrow \infty$ cannot be exchanged. 
	
	Moreover, the main results in this paper provide a guideline for numerical approximation for $V^{Q^\infty}(x,f^\infty)$: With good approximations of the transition kernel $Q^\infty$ and reward function $f^\infty$, taking supremum only over equilibria may not provide good estimation for the target optimal value. Instead, one should take supremum over all $\varepsilon$-equilibria.
\end{Remark}

\begin{Remark}
	Analogous to Remark \ref{rm.sets.semicontinue}, if the same conditions in Theorem \ref{thm.Qf.continuity} hold, then
	$$
	\lim_{\eps\searrow 0}\left(\liminf_{n\to \infty} \Ec^{Q^n}_\eps (f^n)\right)=\lim_{\eps\searrow 0}\left(\limsup_{n\to \infty} \Ec^{Q^n}_\eps (f^n)\right)=\Ec^{Q^\infty} (f^\infty).
	$$
	\begin{proof}
		By a similar argument as in Remark \ref{rm.sets.semicontinue}, we can show that
		$$
		\lim_{\eps\searrow 0}\left(\limsup_{n\to \infty} \Ec^{Q^n}_\eps (f^n)\right)\subset\Ec^{Q^\infty} (f^\infty).
		$$
		It remains to show that 
		\be\label{eq.sets} 
		\Ec^{Q^\infty} (f^\infty)\subset \lim_{\eps\searrow 0}\left(\liminf_{n\to \infty} \Ec^{Q^n}_\eps (f^n)\right).
		\ee 
		For $S\in \Ec^{Q^\infty} (f^\infty)$, we have
		$$
		\begin{cases}
			f^{\infty}(x)\leq \E^{Q^{\infty}}[\delta(\rho(S)f^{\infty}(X_{\rho(S)}))],\quad \forall x\notin S;\\
			f^{\infty}(x)\geq \E^{Q^{\infty}}[\delta(\rho(S)f^{\infty}(X_{\rho(S)}))],\quad \forall x\in S.
		\end{cases}
		$$
		Then for any $\varepsilon>0$, Lemma \ref{lm.tau.uniform} implies that, for $n$ big enough,
		$$
		\begin{cases}
			f^{n}(x)-\varepsilon\leq \E^{Q^{n}}[\delta(\rho(S)f^{n}(X_{\rho(S)}))],\quad \forall x\notin S;\\
			f^{n}(x)+\varepsilon\geq \E^{Q^{n}}[\delta(\rho(S)f^{n}(X_{\rho(S)}))],\quad \forall x\in S.
		\end{cases}.
		$$
		Consequently, $S\in\liminf_{n\to \infty} \Ec^{Q^n}_\eps (f^n)$ for any $\eps>0$, which implies \eqref{eq.sets}.
	\end{proof}
\end{Remark}

The following example shows that the continuity result in Theorem \ref{thm.Qf.continuity} may fail if the convergence of $(Q_n)_{n\in \N}$ in total variation is only assumed to be locally uniform instead of uniform.

\begin{Example}\label{ex:countersecond}
	Let $\X=\{y, x_0, x_1,x_2,...\}\subset \R$.  Define
	\begin{align*}
		&Q^n:
		\begin{cases}
			Q^n(x_i, x_{i+1})=\frac{1}{2}, Q^n(x_i,y)=\frac{1}{2},\quad & 0\leq i< n,\\
			Q^n(x_i,y)=1,\quad &  i> n\\
			Q^n(x_n,x_n)=1, Q^n(y,y)=1.
		\end{cases};\\
	&Q^\infty:
	\begin{cases} 
		Q^\infty(x_i, x_{i+1})=\frac{1}{2}, Q^\infty(x_i, y)=\frac{1}{2}, \quad \forall i\geq 0,\\
		Q^\infty(y, y)=1. 
	\end{cases}
	\end{align*}
One can easily see that $Q^n$ converges to $Q^\infty$ locally uniformly, but not uniformly. Let $f(x_i)=1$ for $i\in\N$, $f(y)=2.99$, and $\delta(k)=\frac{1}{1+k}$ for $k\in\N$. 

We have 
	$\frac{1}{2}\delta(1)(1+f(y))=\frac{3.99}{4}<1$, and 
	$$
	\sum_{k=1}^\infty \delta(k)\left(\frac{1}{2}\right)^k f(y)>\sum_{k=1}^3 \delta(k)\left(\frac{1}{2}\right)^k f(y)=2.99\left(\frac{1}{4}+\frac{1}{12}+\frac{1}{32}\right)>1.
	$$ 
	That is,
	\be\label{eq.eg.M}
	\frac{1}{2}\delta(1)(1+f(y))<1<\sum_{k=1}^\infty \delta(k)\left(\frac{1}{2}\right)^k f(y).
	\ee
Take $\eps$ with $0<\varepsilon<1-\frac{1}{2}\delta(1)(1+f(y))$. For any $n<\infty$ and $S\in \Ec^{Q^n}(f,\varepsilon)$, it is easy to check that $y,x_n\in S$. For any $i\leq n$, if $x_i\in S$, then by the first inequality in \eqref{eq.eg.M}, $x_{i-1}\in S$. Hence, for any $n<\infty$,
$$
\{x_0,x_1,...,x_n\}\subset S,\quad \forall S\in \Ec^{Q^n}(f,\varepsilon).
$$
As the above holds for any $\eps$ with $0<\varepsilon<1-\frac{1}{2}\delta(1)(1+f(y))$, we have that
$$
\limsup_{n\rightarrow \infty}V_\varepsilon^n(x_0)=f(x_0),\quad \forall n<\infty, 
$$
which leads to 
\be\label{eq.eg.x0}
\limsup_{\varepsilon\searrow 0}\limsup_{n\rightarrow \infty}V_\varepsilon^n(x_0)=f(x_0).
\ee
On the other hand, the second inequality in \eqref{eq.eg.M} indicates $J^\infty(x_i, \{y\})>f(x_i)$ for any $i\in\N$. This together with $J^\infty(y, \{y\})<f(y)$ implies that 
$$\hat{S}^\infty=\{y\}\quad\text{and}\quad V^\infty (x_0)=J^\infty(x_0,\{y\})=\sum_{k=1}^\infty \delta(k)\left(\frac{1}{2}\right)^k f(y).$$
Then by \eqref{eq.eg.x0} and the second inequality in \eqref{eq.eg.M},
$$
\limsup_{\varepsilon\searrow 0}\limsup_{n\rightarrow \infty}V_\varepsilon^n(x_0)<V^\infty(x_0).
$$
\end{Example}

However, if we use $W^{Q^n}_\varepsilon(.,f^n)$ (instead of $V^{Q^n}_\varepsilon(.,f^n)$) to approximate $V^{Q^\infty}(.,f^\infty)$, then we can weaken the uniform convergence in total variation condition to locally uniform convergence as shown in the following proposition.

\begin{Proposition}\label{prop.continue.pseudo}
	Suppose the conditions for $(f^n)_{n\in \overline \N}$ and $\delta$ in Theorem \ref{thm.Qf.continuity} hold, and $Q^n$ converges to $Q^\infty$ locally uniformly in total variation. Assume that for any compact set $K$ and $\eps>0$, there exists a compact set $K'$ such that $\sup_{x\in K} Q^\infty(x, K')\geq 1-\varepsilon$.
	Then
	\begin{align*}
		& \lim\limits_{\varepsilon\searrow 0}\Big( \liminf_{n\rightarrow \infty} W^{Q^n}_\varepsilon(x,f^n)\Big)
		= \lim\limits_{\varepsilon\searrow 0}\Big( \limsup_{n\rightarrow \infty} W^{Q^n}_\varepsilon(x,f^n)\Big)
		=V^{Q^\infty}(x,f^\infty),\quad \forall x\in \X.
	\end{align*}
\end{Proposition}
The proof of Proposition \ref{prop.continue.pseudo} is presented in Section \ref{subsec:proof.continue}

\subsection{Proofs of Theorem \ref{thm.Qf.continuity} and Proposition \ref{prop.continue.pseudo}}\label{subsec:proof.continue}

To prepare for the proofs of Theorem \ref{thm.Qf.continuity} and Proposition \ref{prop.continue.pseudo}, we first provide some auxiliary results for (pseudo) $\varepsilon$-equilibria.

\begin{Lemma}\label{lm.equi.pseduvs}
	Fix a bounded reward function $f$ and a transition kernel $Q$. We have that
	$$
	\Ec^Q(f,\varepsilon)\subset \Gc^Q(f,\varepsilon),\quad \forall \varepsilon\geq 0,
	$$
	and 
	$$
	V^Q_\varepsilon(x,f)\leq W^Q_\varepsilon(x,f),\quad \forall x\in \X,\forall \varepsilon\geq 0.
	$$
\end{Lemma}

\begin{proof}
 The result directly follows from Definitions \ref{def.equi.epsi} and \ref{def.equi.pseudoepsi}.
\end{proof}

%Write $W^n, W^n_\varepsilon, V^n_\varepsilon,$ short for $W^{Q^n}, W^{Q^n}_\varepsilon$, $V^{Q^n}_\varepsilon$ respectively for the consistency with notation $V^n$, for all $n\in \overline \N$. And we continue to use the notations of $\hat{S}^n, \forall n\in \overline \N$ for the smallest optimal  equilibrium under $Q^n, \forall n\in \overline \N$.

\begin{Lemma}\label{lm.equi.pseudo}
	Let Assumption \ref{assume.delta} hold. Let $f\geq 0$ be a bounded reward function and $Q$ be a transition kernel. 
	\bi 
	\item[(a)]
	Given $S,T\in \Gc^Q(f)$, we have that
	$
	S\cap T\in \Gc^Q(f).
	$
	\item[(b)]
	Let $S,R\in\Bc$ such that $S\in \Gc^Q(f)$ and $R\supset S$. Then
	$$
	J^Q(x,S,f)\geq J^Q(x,R,f),\quad \forall x\in \X.
	$$
	\ei
\end{Lemma}

\begin{proof}
	Part (a):
	We can use the same argument as that in the proof of \cite[lemma 4.1]{MR3911711} to get that
	$$
	J(x,S\cap T)\geq J(x, S)\vee J(x,T)\geq f(x), \quad\forall\, x\notin S\cap T,
	$$
	which implies $S\cap T\in \Gc^Q(f)$. 
	
	Part (b): Notice that $J^Q(x,S,f)=f(x)=J^Q(x,R,f)$, for all $x\in S$. For $x\notin S$, same discussion in the proof of \cite[Lemma 3.1]{MR4116459} (or \cite[Lemma 4.1]{MR4250561} ) can be applied to reach that
	$$
	J^Q(x, S,f)\geq J^Q(x,R,f).
	$$
\end{proof}
Define
$$S_*(f,Q):=\cap_{s\in \Gc^Q(f)} S.$$
Recall the smallest optimal equilibrium, $S^*(f,Q)=\cap_{s\in \Ec^Q(f)} S$ defined in \eqref{e001}. The following proposition shows that $S_*(f,Q)$ is optimal among all pseudo equilibria and also coincides with $S^*(f,Q)$.

\begin{Proposition}\label{prop.optimalequi.pseudo}
	Let Assumption \ref{assume.delta} hold. Given a bounded reward function $f\geq 0$ and a transition kernel $Q$, we have that
	$$S_*(f,Q)=S^*(f,Q)\quad\text{and}\quad W^Q(x,f)=J^Q(x,S_*(f,Q),f)=V^Q(x,f),\ \forall x\in \X.
	$$
\end{Proposition}

\begin{proof}%[Proof of Proposition \ref{prop.optimalequi.pseudo}]
	By Lemma \ref{lm.equi.pseduvs}, $\Ec^Q(f)\subset \Gc^Q(f)$ and thus $S_*(f,Q)\subset S^*(f,Q)$. We show $S^*(f,Q)\subset S_*(f,Q)$ by the iterative construction for $S^*(f,Q)$. Recall $S^*(f,Q)=\cup_{n\in \N} S_n$ in Lemma \ref{lm.iteration.sstar}, where $(S_n)_{n\in \N}$ is an increasing sequence defined as $S_0=\emptyset$, and
	$$
	S_{n+1}:=\{x\in \X\setminus S_n: f(x)>\sup_{S:S_n\subset S\subset \X\setminus\{x\}} J^Q(x,S,f) \}, \quad n\in\N.
	$$ 
	For any $R\in \Gc^Q(f)$, we prove by induction that
	\begin{equation}\label{e002}
	S_n\subset R,\quad\forall\,n\in \N.
	\end{equation}
We have $S_0=\emptyset\subset R$. Suppose $S_n\subset R$, then for any $x\notin R$, 
	$$
	f(x)\leq J^Q(x,R,f)\leq \sup_{S:S_n\subset S\subset \X\setminus\{x\}} J^Q(x,S,f), 
	$$
	and thus $x\notin S_{n+1}$. Therefore, $S_{n+1}\subset R$. 
	
By \eqref{e002}, $S^*(f,Q)=\cup_{n\geq 0}S_n\subset R$ for any $R\in \Gc^Q(f)$, which implies $S^*(f,Q)\subset S_*(f,Q)$. Hence, $S_*(f,Q)=S^*(f,Q)$. Moroever, for any $S\in \Gc^Q(f)$, by Lemma \ref{lm.equi.pseudo} part (b), 
	$$
	J^Q(x,S_*(f,Q),f)\geq J^Q(x,S,f),\quad \forall x\in \X,
	$$
	so $J^Q(.,S_*(f,Q),f)=W^Q(.,f)$. Together with Lemma \ref{lm.iteration.sstar}, we have that
	$$
	W^Q(x,f)=J^Q(x,S_*(f,Q),f)=J^Q(x,S^*(f,Q),f)=V^Q(x,f),\quad \forall x\in \X.
	$$
\end{proof}

\begin{Lemma}\label{lm.f.epsilon}
	Suppose Assumption \ref{assume.delta} holds. For any $0\leq \varepsilon_1\leq \varepsilon_2$, we have that
	\be\label{eq.fepsilon.equiset}  
	\Gc^Q((f-\varepsilon_1)\vee 0)\subset \Gc^Q((f-\varepsilon_2)\vee 0).
	\ee
	Therefore,
	\be\label{eq.fepsilon.optimalequi}    
	S_*((f-\varepsilon_1)\vee 0,Q)\supseteq S_*((f-\varepsilon_2)\vee 0,Q).
	\ee 
\end{Lemma}

\begin{proof}
	Let $S\in \Gc^Q(f-\varepsilon_1)$. For any $x\notin S$,
	\begin{align*}
		&|J^Q(x,S,(f-\varepsilon_1)\vee 0)-J^Q(x,S,(f-\varepsilon_2)\vee 0)| \\
		= & \E^x\left[\delta(\rho(S)) \left(\left(\left(f\left(X_{\rho(S)}\right)-\varepsilon_1\right)\vee 0\right)-\left(\left(f\left(X_{\rho(S)}\right)-\varepsilon_2\right)\vee 0\right)\right)\right]\\
		\leq & \E^x[\delta(\rho(S))(\varepsilon_2-\varepsilon_1)]
		\leq \varepsilon_2-\varepsilon_1.
	\end{align*}
	If $f(x)\geq \varepsilon_2$, then
	\begin{align*}
		&J^Q(x,S,(f-\varepsilon_2)\vee 0)\geq J^Q(x,S,(f-\varepsilon_1)\vee 0)-(\varepsilon_2-\varepsilon_1)\\
		\geq &f(x)-\varepsilon_1-(\varepsilon_2-\varepsilon_1)=f(x)-\varepsilon_2,
	\end{align*}
	where the second inequality follows that $S\in \Gc^Q(f-\varepsilon_1)$. 	
	If $f(x)<\varepsilon_2$, then $J^Q(x,S,(f-\varepsilon_2)\vee 0)\geq 0=(f(x)-\varepsilon_2)\vee 0$. Hence, $S\in \Gc^Q(f-\varepsilon_2)$.
	% and the proof for \eqref{eq.fepsilon.equiset} is complete.
%	
%	By \eqref{eq.fepsilon.equiset} and Lemma \ref{lm.iteration.sstar}, 
%	$$
%	S_*((f-\varepsilon_2)\vee 0,Q)= \cap_{S\in \Gc^Q((f-\varepsilon_2)\vee 0)} S\subset  \cap_{S\in \Gc^Q((f-\varepsilon_1)\vee 0)} S=S_*((f-\varepsilon_1)\vee 0,Q).
%	$$
\end{proof}

%The following is a monotone result of optimal (pseudo) equilibrium and the optimal value function w.r.t the reward function.

\begin{Lemma}\label{lm.optimal.pseudo}
	Suppose Assumption \ref{assume.delta} holds. Given a bounded reward function $f\geq 0$ and a transition kernel $Q$, we have that
%	\be\label{eq.continuef.equiset} 
%	\Gc^Q((f-\varepsilon)\vee 0) \downarrow\Gc^Q(f),\quad \text{as}\; \varepsilon \searrow 0,
%	\ee 
	\be\label{eq.continuef.optimalequi} 
	S^*((f-\varepsilon)\vee 0,Q)=S_*((f-\varepsilon)\vee 0,Q) \uparrow S_*(f,Q)=S^*(f,Q), \quad \text{as}\; \varepsilon \searrow 0,
	\ee 
	and
	\be\label{eq.continuef.value}  
	\lim_{\varepsilon \searrow 0}V^Q(x,(f-\varepsilon)\vee 0)=V^Q(x,f),\quad \forall x\in \X.
	\ee 
\end{Lemma}	

\begin{proof}
%	By Lemma \ref{lm.f.epsilon}, $\Gc^Q((f-\varepsilon)\vee 0)$ decreases as $\varepsilon$ decreases, so we only need to prove $\cap_{\varepsilon>0} \Gc^Q((f-\varepsilon)\vee 0)\subset \Gc^Q(f)$. Let $S\in \cap_{\varepsilon>0} \Gc^Q((f-\varepsilon)\vee 0)$. For $x\notin S$,
%	we have
%	$$
%	\E^x[\delta(\rho(S))((f(X_{\rho(S)})-\varepsilon)\vee 0)]\geq (f(x)-\varepsilon)\vee 0,\quad \forall \varepsilon>0,
%	$$
%	taking $\varepsilon\searrow0$ we get that $\E^x[\delta(\rho(S))(f(X_{\rho(S)}))]\geq f(x)$. Hence, $S\in \Gc^Q(f)$ and \eqref{eq.continuef.equiset} is established.
	
	As for \eqref{eq.continuef.optimalequi}, by Lemma \ref{lm.f.epsilon}, $S_*((f-\varepsilon)\vee 0,Q)$ increases as $\varepsilon\searrow0$, so 
	$$
	S':=\cup_{\varepsilon>0 } S_*((f-\varepsilon)\vee 0,Q)\subset S_*(f,Q).
	$$ 
	Given $x\notin S'$,
	\begin{align*}
		\E^x[\delta(\rho(S'))f(X_{\rho(S')})]=&\lim\limits_{\varepsilon\searrow 0}E^x[\delta(\rho(S_*((f-\varepsilon)\vee 0,Q)))((f(X_{\rho(S_*((f-\varepsilon)\vee 0,Q))})-\varepsilon)\vee 0)]\\
		=&\lim\limits_{\varepsilon\searrow 0}J^Q(x,S_*((f-\varepsilon)\vee 0,Q),(f-\varepsilon)\vee 0)
		\geq \lim\limits_{\varepsilon\searrow 0} (f(x)-\varepsilon)\vee 0\\
		=&f(x),
	\end{align*}
	where the second line follows that $x\notin S_*((f-\varepsilon)\vee 0,Q)$. Hence, $S'\in \Gc^Q(f)$ and $S_*(f,Q)\subset S'$, which implies $S'=S_*(Q,f)$. Then by Proposition \ref{prop.optimalequi.pseudo}, 
	$$
	S^*((f-\varepsilon)\vee 0,Q) =S_*((f-\varepsilon)\vee 0,Q) \uparrow S_*(f,Q)=S^*(f,Q) ,\quad \text{as}\; \varepsilon \searrow 0.
	$$
	
	Now we prove \eqref{eq.continuef.value}. By \eqref{eq.continuef.optimalequi}, for $x\in S^*(f,Q)$, $x\in 	S^*((f-\varepsilon)\vee 0,Q)$ for $\varepsilon$ small enough, and thus
	$$\lim_{\varepsilon \searrow 0}V^Q(x,(f-\varepsilon)\vee 0)=\lim_{\varepsilon \searrow 0}(f(x)-\varepsilon)\vee 0=f(x)=V^Q(x,f),\quad \forall x\in S_*(f,Q).$$
	For $x\notin S_*(f,Q)$, by  \eqref{eq.continuef.optimalequi}, $\rho(S^*(f-\varepsilon)\vee 0,Q)\rightarrow \rho(S^*(f,Q))$ a.s. and $(f-\varepsilon)\vee 0\rightarrow f$ as $\varepsilon\searrow 0$. Then by Dominated Convergence Theorem,
	\begin{align*}
		\lim\limits_{\varepsilon\searrow 0}V^Q(x,(f-\varepsilon)\vee 0)=&\lim\limits_{\varepsilon\searrow 0} E^x[\delta(\rho(S^*((f-\varepsilon)\vee 0,Q)))((f(X_{\rho(S^*((f-\varepsilon)\vee 0,Q))})-\varepsilon)\vee 0)]\\
		=& \E^x[\delta(\rho(S^*(f,Q)))f(X_{\rho(S^*(f,Q))})]=V^Q(x,f),\quad x\notin S_*(f,Q).
	\end{align*}
	which completes the proof of \eqref{eq.continuef.value}.
\end{proof}

\begin{Lemma}\label{lm.equi.epsilon}
	Suppose Assumption \ref{assume.delta} holds. Let $f\geq 0$ be a bounded reward function and $Q$ be a transition kernel. Then for any $\varepsilon>0$, we have that
%	\bi 
%	\item [(a)]
%	For any $\varepsilon>0$, 
	$$
	\Gc^Q(f)\subset  \Gc^Q((f-\varepsilon)\vee 0)\subset \Gc^Q(f,\varepsilon)\subset \Gc^Q\left(\left(f-\frac{\varepsilon}{1-\delta(1)}\right)\vee 0\right).
	$$
%	\item [(b)]
%	For $0\leq \varepsilon_1\leq \varepsilon_2$, we have
%	$
%	\Gc^Q(f,\varepsilon_1)\subset \Gc^Q(f,\varepsilon_2). 
%	$
%	\item [(c)] $\Gc^Q(f)=\cap_{\varepsilon>0} \Gc^Q(f,\varepsilon)$.
%	\ei 
\end{Lemma} 

\begin{proof}
%	Part (a): 
	$\Gc^Q(f)\subset \Gc^Q((f-\varepsilon)\vee 0)$ follows Lemma \ref{lm.f.epsilon}.
	
	Let $S\in \Gc^Q((f-\varepsilon)\vee 0)$. For any $x\notin S$, if $f(x)\geq \varepsilon$, then
	$$
	\E_x^Q[\delta(\rho(S)) f(X_{\rho(S)})]\geq \E_x^Q[\delta(\rho(S)) ((f(X_{\rho(S)})-\varepsilon)\vee 0)]\geq (f(x)-\varepsilon)\vee 0=f(x)-\varepsilon.
	$$
	If $f(x)<\varepsilon$, obviously, $	\E_x^Q[\delta(\rho(S)) f(X_{\rho(S)})]\geq 0>f(x)-\varepsilon$. So $S\in \Gc^Q(f,\varepsilon)$. 
	
	Let $S\in \Gc^Q(f,\varepsilon)$. Take $x\notin S$. If $f(x)\geq \frac{\varepsilon}{1-\delta(1)}$, then by $\rho(S)\geq 1$ we have that
	\begin{align*}
		\E_x^Q\left[\delta(\rho(S)) \left(\left(f(X_{\rho(S)})-\frac{\varepsilon}{1-\delta(1)}\right)\vee 0\right)\right]\geq &\E_x^Q[\delta(\rho(S)) f(X_{\rho(S)})]-\delta(1)\cdot \frac{\varepsilon}{1-\delta(1)}\\
		\geq  & f(x)-\varepsilon-\frac{\delta(1)\varepsilon}{1-\delta(1)}
		= f(x)-\frac{\varepsilon}{1-\delta(1)},
	\end{align*}
	where the second line follows from $S\in \Gc^Q(f,\varepsilon)$. If $f(x)<\frac{\varepsilon}{1-\delta(1)}$, then $$
	\E_x^Q\left[\delta(\rho(S)) \left(\left(f(X_{\rho(S)})-\frac{\varepsilon}{1-\delta(1)}\right)\vee 0\right)\right]\geq0=\left(f(x)-\frac{\varepsilon}{1-\delta(1)}\right)\vee 0.
	$$
	Hence, $S\in \Gc^Q((f-\frac{\varepsilon}{1-\delta(1)})\vee 0)$.
	
%	Part (b): Let $S\in \Gc^Q(f,\varepsilon_1)$. For $x\notin S$, $f(x)\leq J(x,S)+\varepsilon_1\leq J(x,S)+\varepsilon_2$. Hence, $S\in \Gc^Q(f,\varepsilon_2)$.
	
%	Part (c): Let $S\in \cap_{\varepsilon>0}\Gc^Q(f,\varepsilon)$. For $x\notin S$,
%	$$
%	\E_x^Q[\delta(\rho(S)) f(X_{\rho(S)})]\geq f(x)-\varepsilon,\quad \forall \varepsilon>0,
%	$$
%	and thus $\E_x^Q[\delta(\rho(S)) f(X_{\rho(S)})]\geq f(x)$. Therefore, $S\in \Gc^Q(f)$
%	and consequently $\cap_{\varepsilon>0}\Gc^Q(f,\varepsilon)\subset \Gc^Q(f)$.
%	The converse follows Part (b).
\end{proof}
\begin{proof}[{\bf \textit{Proof of Theorem \ref{thm.Qf.continuity}}}]
The proof is a combination of the following two steps.

Step 1.
%Let the assumptions in Theorem \ref{thm.Qf.continuity} hold.
%Then for any $\varepsilon>0$, there exists $N(\varepsilon)\in\N$ such that $S^*(f^\infty,Q^\infty)\in \Ec^{Q^n}_\varepsilon(f^n)$ for all $n\geq N(\varepsilon)$. Moreover, 
We first prove, under assumptions in Theorem \ref{thm.Qf.continuity}, that
\be\label{eq.prop.liminf} 
	V^{Q^\infty}(x,f^\infty)\leq \liminf_{n\rightarrow \infty} V^{Q^n}_\varepsilon(x,f^n)\leq   \liminf_{n\rightarrow \infty} W^{Q^n}_\varepsilon(x,f^n), \quad \forall \varepsilon>0.
\ee	
		Let $\varepsilon>0$. Applying Lemma \ref{lm.tau.uniform}(c) with $\tau=\rho(S^*(f^\infty,Q^\infty))$, there exists $N\in \N$ such that 
		$$
\sup_{x\in \X} |v^{Q^n}(x, \rho(S^*(f^\infty,Q^\infty)), f^n)-v^{Q^\infty}(x, \rho(S^*(f^\infty,Q^\infty)), f^\infty)|\leq\varepsilon.
	$$
	Then
\begin{align*}
	v^{Q^n}(x, \rho(S^*(f^\infty,Q^\infty)), f^n)&\geq v^{Q^\infty}(x, \rho(S^*(f^\infty,Q^\infty)), f^\infty)-\varepsilon\geq f(x)-\varepsilon,\quad \forall x\notin S^*(f^\infty,Q^\infty),\\
v^{Q^n}(x, \rho(S^*(f^\infty,Q^\infty)), f^n)&\leq v^{Q^\infty}(x, \rho(S^*(f^\infty,Q^\infty)), f^\infty)+\varepsilon\leq f(x)+\varepsilon,\quad \forall x\in S^*(f^\infty,Q^\infty).
\end{align*}
	Hence, $S^*(f^\infty,Q^\infty)\in \Ec^{Q^n}_\varepsilon(f^n)$ for all $n\geq N$.
	
Now take $x\in \X$. For $n\geq N$, by Definition \ref{def.equi.pseudoepsi} and \eqref{eq.def.WVepsilon},
	$$
	V^{Q^n}_\varepsilon(x,f^n)\geq J^{Q^n}(x,S^*(f^\infty,Q^\infty),f^n),%quad \forall x \in \X,
	$$
	which leads to
$$
 \liminf_{n\rightarrow \infty} V^{Q^n}_\varepsilon(x,f^n)\geq \liminf_{n\rightarrow \infty} J^{Q_n}(x,S^*(f^\infty,Q^\infty),f^n) =V^{Q^\infty}(x,f^\infty),%\quad \forall x \in \X,
$$
	where the second (in)equality follows from Lemma \ref{lm.tau.uniform}(a).
	By Lemma \ref{lm.equi.pseduvs}, $W^{Q^n}_\varepsilon(x,f^n)\geq V^{Q^n}_\varepsilon(x,f^n)$, and Step 1 is completed.

Step 2. Now we show, under the same assumptions in Theorem \ref{t1} (which are weaker than the assumptions in Theorem \ref{thm.Qf.continuity}), that
%\subsection{Lower semi-continuity using $\varepsilon$-pseudo  equilibria}
%\begin{Proposition}\label{prop.continue.upper}
%Let the assumptions in Theorem \ref{t1} hold. Then
	\be\label{eq.continue.upper}  
	\lim\limits_{\varepsilon\searrow 0} \left(\limsup_{n\rightarrow \infty} V^{Q^n}_\varepsilon(x,f^n)\right) \leq  	\lim\limits_{\varepsilon\searrow 0} \left(\limsup_{n\rightarrow \infty} W^{Q^n}_\varepsilon(x,f^n)\right)\leq V^{Q^\infty}(x,f^\infty),\quad \forall x\in \X.
	\ee 
%		Same inequality holds if exchanging the order of taking limits: $n\rightarrow\infty$ and $m\rightarrow \infty$.
%\end{Proposition}
%\begin{proof}
	By  Theorem \ref{t1} and Proposition \ref{prop.optimalequi.pseudo}, for any $\varepsilon\geq 0$,
	\be\label{eq0} 
	\begin{aligned}
\limsup_{n\rightarrow \infty}V^{Q^n}\left(x,\left(f^n-\frac{\varepsilon}{1-\delta(1)}\right)\vee 0\right)=& \limsup_{n\rightarrow \infty}W^{Q^n}\left(x,\left(f^n-\frac{\varepsilon}{1-\delta(1)}\right)\vee 0\right)\\
	\leq & V^{Q^\infty}\left(x,\left(f^\infty-\frac{\varepsilon}{1-\delta(1)}\right)\vee 0\right),\quad \forall x\in \X.
	\end{aligned}
	\ee 
	Meanwhile, for $n\in\overline \N$, 
	\be\label{eq1} 
	\begin{aligned}
		V^{Q^n}_\varepsilon(x,f^n)\leq & W^{Q^n}_\varepsilon(x,f^n)\\
		\leq & W^{Q^n}\left(x,\left(f^n-\frac{\varepsilon}{1-\delta(1)}\right)\vee 0\right)+\frac{\varepsilon}{1-\delta(1)}\\
		=&V^{Q^n}\left(x,\left(f^n-\frac{\varepsilon}{1-\delta(1)}\right)\vee 0\right)+\frac{\varepsilon}{1-\delta(1)},\quad \forall x\in \X.
	\end{aligned}
	\ee 
	where the first line follows from Lemma \ref{lm.equi.pseduvs}, the second line follows from $\Gc^{Q^n}(f^n,\varepsilon)\subset \Gc^{Q_n}((f^n-\frac{\varepsilon}{1-\delta(1)})\vee 0)$ implied by Lemma \ref{lm.equi.epsilon}, and the last line follows from Proposition \ref{prop.optimalequi.pseudo}.
	By \eqref{eq0} and \eqref{eq1}, for any $\varepsilon\geq 0$ and $x\in\X$,
\begin{align*}
\limsup_{n\rightarrow \infty} V^{Q^n}_\varepsilon(x,f^n)\leq\limsup_{n\rightarrow \infty} W^{Q^n}_\varepsilon(x,f^n)\leq & \limsup\limits_{n\rightarrow \infty} V^{Q^n}\left(x,\left(f^n-\frac{\varepsilon}{1-\delta(1)}\right)\vee 0\right)+\frac{\varepsilon}{1-\delta(1)}\\
	 \leq & V^{Q^\infty}\left(x,\left(f^\infty-\frac{\varepsilon}{1-\delta(1)}\right)\vee 0\right)+\frac{\varepsilon}{1-\delta(1)},
\end{align*}
	Then \eqref{eq.continue.upper} follows by setting $Q=Q^\infty$ in \eqref{eq.continuef.value}.
%\end{proof}

\end{proof}

\begin{proof}[{\bf \textit{Proof of Proposition \ref{prop.continue.pseudo}}}]
	Step 1. Let $\eps>0$. We first prove that for any $x\in \X\setminus S^*(f^\infty,Q^\infty)$, there exists a set $S_x$ and $N\in\N$ such that $$
	S_x\in \Gc^{Q^n}(f^n,\varepsilon),\quad
	J^{Q^n}(x,S^*(f^\infty,Q^\infty),f^n)\leq J^{Q^n}(x,S_x,f^n)+\varepsilon\quad\text{and}  \quad \forall n\geq N.
	$$
	Fix $x\notin S^*(f^\infty,Q^\infty)$. As $\sup_{n\in \overline \N} \|f^n\|_{\infty}=:M<\infty$, we can take $T\in \N$ such that $\delta(T)M<\varepsilon/2$. Then we apply the same discussion as \eqref{eq.n.T} to find a compact set $K$ and $N_1\in\N$ (that may depend on $x$) such that 
\be\label{eq.local.thm}
	2M\left(1-\P^{Q^n}_x(X_t\in K,\ t=0,\dotso,T)\right)=2M\cdot \P^{Q^n}_x(\rho(\X\setminus K)\leq T)<\varepsilon/2, \quad \forall N_1\leq n\leq \infty.
\ee
	 By Lemma \ref{lm.tau.uniform}(b),
	$$
\lim_{n\rightarrow\infty } \sup_{y\in (K\setminus S^*(f^\infty, Q^\infty))}|J^{Q^\infty}(y,S^*(f^\infty, Q^\infty),f^\infty)-J^{Q^n}(y,S^*(f^\infty, Q^\infty),f^n)|=0.
	$$
	This together with the locally uniform convergence of $(f^n)_{n\in \N}$, we can find $N_2\in \N$ (that may depend on $x$) such that, for all $n\geq N_2$, $\sup_{y\in K}|f^n(y)-f^\infty(y)|<\frac{\varepsilon}{2}$ and
	$$
J^{Q^\infty}(y,S^*(f^\infty, Q^\infty))-\frac{\varepsilon}{2}\leq J^{Q^n}(y,S^*(f^\infty,Q^\infty),f^\infty),\quad \forall y\in (K\setminus S^*(f^\infty, Q^\infty)).
$$
This imply that for all $n\geq N_2$,
	\begin{align}
\notag		f^n(y)-\varepsilon \leq & f^\infty(y)-\frac{\varepsilon}{2}\leq J^{Q^\infty}(y,S^*(f^\infty,Q^\infty),f^\infty)-\frac{\varepsilon}{2}\\
\label{e003}		\leq &J^{Q^n}(y,S^*(f^\infty,Q^\infty),f^n),\quad \forall y\in (K\setminus S^*(f^\infty, Q^\infty)).
	\end{align}
Let 
$$S_x:=S^*(f^\infty,Q^\infty)\cup (\X\setminus K).$$
By \eqref{e003}, $S_x\in \Gc^{Q^n}_\varepsilon(f^n)$ for $n\geq N_2$. Moreover, for any $n\geq N:=N_1\vee N_2$,
	\begin{align*}
		&|J^{Q^n}(x,S^*(f^\infty,Q^\infty),f^n)-J^{Q^n}(x,S_x,f^n)|\\
		\leq & \E^{Q^n}_x[|\delta(\rho(S^*(f^\infty,Q^\infty)))f^n(X_{\rho(S^*(f^\infty,Q^\infty))})-\delta(\rho(S_x))f^n(X_{\rho(S_x)})|\cdot 1_{\{X_{\rho(S_x)}\notin S^*(f^\infty,Q^\infty), \rho(S_x)\geq T\}}]\\
		&+\E^{Q^n}_x[|\delta(\rho(S^*(f^\infty,Q^\infty)))f^n(X_{\rho(S^*(f^\infty,Q^\infty))})-\delta(\rho(S_x))f^n(X_{\rho(S_x)})|\cdot 1_{\{X_{\rho(S_x)}\notin S^*(f^\infty,Q^\infty), \rho(S_x)< T\}}]\\	
		\leq & 2M \delta(T)+2M\cdot \P^{Q^n}_x(\rho(\X\setminus K)\leq T)\\
		<&\varepsilon,
	\end{align*}
where the last line follows from \eqref{eq.local.thm} and $\delta(T)M<\varepsilon/2$. Step 1 is completed.
	
	Step 2. For any $x\notin S^*(f^\infty,Q^\infty)$, we can find $N'\in\N$ (which may depend on $x$) such that $$
	| J^{Q^\infty}(x, S^*(f^\infty,Q^\infty), f^\infty)- J^{Q^n}(x, S^*(f^\infty,Q^\infty), f^\infty)|<\frac{\varepsilon}{2},\quad \forall n\geq N'.
	$$
	 Then from Step 1, 
	\begin{align*}
		V^{Q^\infty}(x,f^\infty)= & J^{Q^\infty}(x, S^*(f^\infty,Q^\infty), f^\infty)\leq J^{Q^n}(x, S^*(f^\infty,Q^\infty), f^n)+\varepsilon\\
		\leq &J^{Q^n}(x, S_x,f^n)+2\varepsilon\leq W^{Q^n}_\varepsilon(f^n)+2\varepsilon,\quad \forall n\geq N\vee N'.
	\end{align*}
	Letting $n\to \infty$ then $\varepsilon\searrow 0$, we have that
	$$
		V^{Q^\infty}(x,f^\infty)\leq \lim_{\eps\searrow 0}\left(\liminf_{n\to \infty} W^{Q^n}_\eps (x,f^n)\right),\quad \forall x\in \X.
	$$ 
Then the rest follows from Step 2 in the proof of Theorem \ref{thm.Qf.continuity}.
\end{proof}

\appendix
\section{Proofs of the lemmata in Section \ref{sec:set.up}}\label{appendix}
\begin{proof}[Proof of Lemma \ref{lm.iteration.sstar}]
Set $S_\infty:=\cup_{k\in \N} S_k$. One  can easily check that same arguments for $S_\infty$ in the proof of Theorem 2 in \cite{MR4205889} is applicable for $S_\infty$.\footnote{The process $X$ is a continuous-time Markov chain in \cite{MR4205889}, while in this paper $X$ is a discrete-time Markov process.} More specifically, Lemmas 2.3, 2.4, 2.5, and the contradiction discussion in the first part of the proof for Theorem 2.2 in \cite{MR4205889} can be applied, and one can obtain an inequality similar as that in \cite[Theorem 2.2]{MR4205889} as follows:
\begin{align*}
	J^Q(y,R,f)-J^Q(y_\infty,S^*(f,Q),f)\leq  \E_y^Q[\delta(\rho(R))]\alpha\leq \delta(1)\alpha<\alpha,
\end{align*}
where the first inequality appears in the proof of \cite[Theorem 2.2]{MR4205889}, and the second inequality follows our time discrete setting. Hence, the same contradiction is reached as that in first part of the proof for  \cite[Theorem 2.2]{MR4205889}, and we have the following:

(i) $S_\infty\subset R,\quad \forall R\in \Ec^Q(f)$;

(ii) For any $S\in \Ec^Q(f)$ and $T\in\Bc$ with $S\subset T$, 
$$
J^Q(x, S,f)\geq J^Q(x,T,f),\quad \forall x\in \X.
$$

(iii) $S_\infty$ is an equilibrium.

By (i) and (iii), $S_\infty=\cap_{S\in \Ec^Q(f)}S=S^*(Q,f)$. Then (ii) implies that $J^Q(x, S^*(Q,f),f)\geq J^Q(x,S,f)$ for any $S\in \Ec^Q(f)$. As a result, $S^*(Q,f)$ is an optimal equilibrium and $V^Q(x,f)=J^Q(x, S^*(Q,f),f)$.
\end{proof}

\begin{proof}[Proof of Lemma \ref{l2}] 
Denote
$$Q_T^n(x,\cdot):=Q^n(x,dx_1)\otimes Q^n(x_1,dx_2)\dotso\otimes Q^n(x_{k-1},dx_k),\quad x\in\X, n\in\overline\N.$$
	Part (a):
Let $\eps>0$. For any $x\in \X$ and compact set $K_0\subset \X$ we have that
\begin{align}
	\notag&Q^n_T(x, (K_0)^T)=\int_{K_0} Q^n(x,dx_1)\int_{K_0} Q^n(x_1,dx_2)\dotso \int_{{K_0}} Q^n(x_{T-1},dx_T)\\
	\notag\geq & \int_{K_0} Q^n(x,dx_1)\dotso \int_{K_0} Q^n(x_{T-2},dx_{T-1})\int_{K_0} Q^\infty(x_{T-1},dx_T)-\sup_{y\in {K_0}}||Q^n(y,.)-Q^\infty(y,.)||_\text{TV}\\
	\notag\geq & \int_{K_0} Q^n(x,dx_1)\dotso\int_{K_0} Q^n(x_{T-3},dx_{T-2}) \int_{K_0} Q^\infty(x_{T-2},dx_{T-1})\int_{K_0} Q^\infty(x_{T-1},dx_T)\\
	\notag&-2\sup_{y\in {K_0}}||Q^n(y,.)-Q^\infty(y,.)||_\text{TV}\\
	\notag\dotso &\\
	\notag\geq &\int_{K_0} Q^\infty(x,dx_1)\dotso\int_{K_0} Q^\infty(x_{T-1},dx_T)-T\sup_{y\in {K_0}}||Q^n(y,.)-Q^\infty(y,.)||_\text{TV}\\
	\label{e2}=&Q^\infty_T(x,(K_0)^T)-T\sup_{y\in {K_0}}||Q^n(y,.)-Q^\infty(y,.)||_\text{TV}.
\end{align}
Exchanging $Q^\infty, Q^n$ in the above inequality and combining with \eqref{e2}, we have
\be\label{eq.tv.k} 
|Q^n_T(x, (K_0)^T)-Q^\infty_T(x, (K_0)^T)|\leq T\sup_{y\in {K_0}}||Q^n(y,.)-Q^\infty(y,.)||_\text{TV},\quad \forall x\in \X.
\ee 

There exists compact subset $K'$ (that may depend on $x$) such that
\be\label{eq.k.inftylocal} 
Q^\infty_T(x,(K')^T)\geq 1-\eps/2.
\ee
By \eqref{eq.tv.k} with $K_0=K'$, there exists $N\in\N$ (that may depend on $K'$) such that 
\be\label{eq.local.uniform}
T\cdot \sup_{y\in K'}||Q^n(y,.)-Q^\infty(y,.)||_\text{TV}\leq \eps/2,\quad \forall n\geq N, 
\ee
This together with \eqref{eq.k.inftylocal} implies that
\be\label{eq.n.T} 
Q^n_T(x,(K')^T)\geq 1-\eps,\quad \forall N\leq n\leq \infty.
\ee
Hence, for any $g\in B(\X^T;[0,1])$, 
\be\label{eq.K.T} 
\left|\E^{Q^n}_x\left[g(X_1,X_2,\dotso,X_T)\right]-\E_x^{Q^n}\left[g(X_1,X_2,\dotso,X_T)\cdot 1_{\{X_t\in K',t=1,\dotso,T\}}\right]\right|\leq\eps,\quad\forall n\geq N.
\ee 
Using a similar argument as that for \eqref{eq.tv.k}, we can show that for any compact set $K_0\subset \X$,
\be\label{eq.g.K} 
\begin{aligned}
	&\left|\E^{Q^n}_x\left[g(X_1,X_2,\dotso,X_T)\cdot 1_{\{X_t\in {K_0},t=1,\dotso,T\}}\right]-\E^{Q^\infty}_x\left[g(X_1,X_2,\dotso,X_T)\cdot 1_{\{X_t\in {K_0},t=1,\dotso,T\}}\right]\right|\\
	\leq & T\sup_{y\in {K_0}}||Q^n(y,.)-Q^\infty(y,.)||_\text{TV},\quad \forall x\in \X.
\end{aligned}
\ee
By \eqref{eq.g.K} with $K_0=K'$ and \eqref{eq.K.T}, 
\be\label{eq.local.gtv} 
\begin{aligned}
	&\sup_{g\in B(\X^T;[0,1])} |\E^{Q^n}_xg(X_1,X_2,\dotso,X_T)-\E^{Q^\infty}_x g(X_1,X_2,\dotso,X_T)|\\
	&\leq 2\eps+T\sup_{y\in K'}||Q^n(y,.)-Q^\infty(y,.)||_\text{TV},\quad \forall n\geq N.
\end{aligned}
\ee
Then the result follows by sending $n\to\infty$ and then $\eps\to 0$.

Part (b): For any $x\in K$, the same discussion from \eqref{eq.local.uniform} to \eqref{eq.g.K} can be applied. Notice that now the compact set $K'$ in \eqref{eq.k.inftylocal} does not depend on $x$ and the integer $N$ in \eqref{eq.local.uniform} only depends on $K'$. Hence, \eqref{eq.local.gtv} is now rewritten as
\begin{align*}
	&\sup_{x\in K, g\in B(\X^T;[0,1])} |\E^{Q^n}_xg(X_1,X_2,\dotso,X_T)-\E^{Q^\infty}_x g(X_1,X_2,\dotso,X_T)|\\
	\leq &2\eps+T\sup_{y\in K'}||Q^n(y,.)-Q^\infty(y,.)||_\text{TV}, \quad \forall n\geq N.
\end{align*}

Part (c):  The same argument from part (a) can be applied and in this case $N$ is independent of $x$. Then we can extend \eqref{eq.local.gtv} to
\begin{align*}
	&\sup_{x\in \X, g\in B(\X^T;[0,1])} |\E^{Q^n}_xg(X_1,X_2,\dotso,X_T)-\E^{Q^\infty}_x g(X_1,X_2,\dotso,X_T)|\\
%	\leq &2\eps+T\sup_{x\in \X, y\in K'}||Q^n(y,.)-Q^\infty(y,.)||_\text{TV}\\
	\leq & 2\eps+ T\sup_{y\in \X}||Q^n(y,.)-Q^\infty(y,.)||_\text{TV},\quad\forall\,n\geq N.
\end{align*}
\end{proof}

\begin{proof}[Proof of Lemma \ref{lm.tau.uniform}]
	Part (a):	Let $\eps>0$. As $M:=\sup_{n\in \overline \N} \|f^n\|_{\infty}<\infty$, there exists $T\in\N$ such that
\be\label{eq.lm.2.3.T}
\sup_{x\in \X, n\in\overline\N, \tau\in \T}\left|v^{Q^n}(x,\tau,f^n)-\E_x^{Q^n}\left[\delta(\tau)f^n(X_\tau)1_{\{\tau\leq T\}}\right]\right|<\eps/4.
\ee
Take $x\in \X$. By Lemma \ref{l2}(a), there exists $N\in \N$ (that may depend on $x$) such that
\be\label{eq.lm2.3.local0} 
\sup_{m\in \overline \N, \tau\in\T}\left|\E_x^{Q^n}\left[\delta(\tau)f^m(X_\tau)1_{\{\tau\leq T\}}\right]-\E_x^{Q^\infty}\left[\delta(\tau)f^m(X_\tau)1_{\{\tau\leq T\}}\right]\right|\leq\eps/4,\quad \forall n\geq N.
\ee
By the locally uniform convergence of $(f^n)_{n\in \overline \N}$, we can first choose a compact set $K'$ (that may depend on $x$) then choose $N'\in\N$ (that may depend on $K'$) such that 
$$
Q^\infty_T(x, (K')^T)\geq 1-\frac{\varepsilon}{16M}\quad\text{and}\quad
\sup_{y\in K'}|f^n(y)-f^\infty(y)|\leq \frac{\varepsilon}{8},\ \forall n\geq N'.
$$
Then 
\be\label{eq.lm2.3.local1}  
\begin{aligned}
	&\sup_{\tau\in\T}\left|\E_x^{Q^\infty}\left[\delta(\tau)f^n(X_\tau)1_{\{\tau\leq T\}}\right]-\E_x^{Q^\infty}\left[\delta(\tau)f^\infty(X_\tau)1_{\{\tau\leq T\}}\right]\right|\\
	\leq & \sup_{\tau\in\T}\left|\E_x^{Q^\infty}\left[\delta(\tau)f^n(X_\tau)1_{\{\tau\leq T\ \text{and}\; X_t\in K', 1\leq t\leq T\}}\right]-\E_x^{Q^\infty}\left[\delta(\tau)f^\infty(X_\tau)1_{\{\tau\leq T\ \text{and}\; X_t\in K', 1\leq t\leq T\}}\right]\right|\\
	&+2\cdot M\cdot \frac{\eps}{16M}\\
	\leq &\sup_{y\in K'}|f^n(y)-f^\infty(y)|+\frac{\eps}{8}\leq \frac{\varepsilon}{4},\quad \forall n\geq N'.
\end{aligned}
\ee
Therefore, by \eqref{eq.lm.2.3.T}-\eqref{eq.lm2.3.local1}, for all $n\geq N\vee N'$,
\be\label{eq.lm2.3.local}  
\begin{aligned}
\sup_{\tau\in\T}|v^{Q^n}(x,\tau,f^n)-v^{Q^\infty}(x,\tau,f^\infty)|\leq  &\sup_{\tau\in \T}\left|v^{Q^n}(x,\tau,f^n)-\E_x^{Q^n}\left[\delta(\tau)f^n(X_\tau)1_{\{\tau\leq T\}}\right]\right|   \\
	&+ 
	\sup_{\tau\in\T}\left|\E_x^{Q^n}\left[\delta(\tau)f^n(X_\tau)1_{\{\tau\leq T\}}\right]-\E_x^{Q^\infty}\left[\delta(\tau)f^n(X_\tau)1_{\{\tau\leq T\}}\right]\right|\\
	&+\sup_{\tau\in\T}\left|\E_x^{Q^\infty}\left[\delta(\tau)f^n(X_\tau)1_{\{\tau\leq T\}}\right]-\E_x^{Q^\infty}\left[\delta(\tau)f^\infty(X_\tau)1_{\{\tau\leq T\}}\right]\right|\\
	&+\sup_{\tau\in \T}\left|v^{Q^\infty}(x,\tau,f^\infty)-\E_x^{Q^\infty}\left[\delta(\tau)f^\infty(X_\tau)1_{\{\tau\leq T\}}\right]\right| \\
	\leq & \eps.
\end{aligned}
\ee 

Part (b): Fix a compact set $K$. By Lemma \ref{l2}(b), we can apply the steps through \eqref{eq.lm2.3.local0}---\eqref{eq.lm2.3.local} by replacing all $\sup_{\tau\in \T}$ (respectively, $\sup_{m\in \overline \N, \tau\in \T}$) with $\sup_{x\in K,\tau\in \T}$ (respectively, $\sup_{x\in K, m\in \overline \N, \tau\in \T}$). Notice that, by assumption on $Q^\infty$, the constants $N, K'$ in this case only depend on $K$ instead of $x$. Hence, the result follows.

Part (c): 
By Lemma \ref{l2}(c), there exists $N>0$ such that
\be\label{eq.lm2.3.unif0} 
\sup_{x\in \X, m\in \overline \N, \tau\in\T}\left|\E_x^{Q^n}\left[\delta(\tau)f^m(X_\tau)1_{\{\tau\leq T\}}\right]-\E_x^{Q^\infty}\left[\delta(\tau)f^m(X_\tau)1_{\{\tau\leq T\}}\right]\right|\leq\eps/4,\quad \forall n\geq N.
\ee
In addition, choose $N'\in\N$ such that $\|f^n-f^\infty\|_{\infty}<\frac{\varepsilon}{4}$ for any $n\geq N'$, 
Then 
\be\label{eq.lm2.3.unif1} 
\sup_{x\in \X, \tau\in\T}\left|\E_x^{Q^\infty}\left[\delta(\tau)f^n(X_\tau)1_{\{\tau\leq T\}}\right]-\E_x^{Q^\infty}\left[\delta(\tau)f^\infty(X_\tau)1_{\{\tau\leq T\}}\right]\right|\leq \|f^n-f^\infty\|_{\infty}<\frac{\varepsilon}{4},\quad \forall n\geq N'.
\ee
Combining \eqref{eq.lm.2.3.T}, \eqref{eq.lm2.3.unif0} and \eqref{eq.lm2.3.unif1}, and replacing ``$\sup_{\tau\in \T}$" with ``$\sup_{x\in \X,\tau\in\T}$" in \eqref{eq.lm2.3.local}, we achieve the desired result.
\end{proof}

\bibliographystyle{plain}
\bibliography{reference}

\end{document}